\documentclass[10pt]{article}

\usepackage{hyperref}
\usepackage{xcolor}
\usepackage{amssymb}
\usepackage{amsfonts}
\usepackage{amsthm}
\usepackage{amsmath}
\usepackage{float}
\usepackage{bigints}
\usepackage{fullpage}
\usepackage{mathtools}
\usepackage[utf8]{inputenc} 
\usepackage{cancel}
\usepackage{verbatim}

\newtheorem{theorem}{Theorem}

\newtheorem{lemma}[theorem]{Lemma}

\newtheorem{definition}{Definition}
\newtheorem{problem}[theorem]{Problem}
\newtheorem{corollary}[theorem]{Corollary}
\newtheorem{remark}[theorem]{Remark}

\DeclareMathOperator{\inter}{int}
\DeclareMathOperator{\bd}{bd}
\DeclareMathOperator{\conv}{conv}
\DeclareMathOperator{\x}{\mathbf{x}}
\DeclareMathOperator{\p}{\mathbf{p}}
\DeclareMathOperator{\q}{\mathbf{q}}
\DeclareMathOperator{\rr}{\mathbf{r}}
\DeclareMathOperator{\e}{\mathbf{e}}

\DeclareMathOperator{\cc}{\mathbf{c}}
\DeclareMathOperator{\aaa}{\mathbf{a}}
\DeclareMathOperator{\bb}{\mathbf{b}}
\DeclareMathOperator{\dd}{\mathbf{d}}
\DeclareMathOperator{\ee}{\mathbf{e}}

\DeclareMathOperator{\y}{\mathbf{y}}
\DeclareMathOperator{\z}{\mathbf{z}}
\DeclareMathOperator{\oo}{\mathbf{o}}
\DeclareMathOperator{\K}{\mathbf{K}}
\DeclareMathOperator{\LL}{\mathbf{L}}
\DeclareMathOperator{\R}{\mathbb{R}}
\DeclareMathOperator{\E}{\mathbb{E}}
\DeclareMathOperator{\B}{\mathbf{B}}

\DeclareMathOperator{\D}{\mathbf{D}}

\DeclareMathOperator{\csep}{c_{\rm sep}}

\DeclarePairedDelimiter\floor{\lfloor}{\rfloor}
\newcommand{\norm}[1]{\left\lVert#1\right\rVert}

\usepackage[pdftex]{graphicx}
\graphicspath{{C:/Users/Sam/Desktop/PaperFigures/}}
\DeclareGraphicsExtensions{.pdf}

\title{On contact graphs of totally separable domains 
\footnote{\hspace*{.35cm} 2010 Mathematics Subject Classification: (Primary) 05C10, 52C15, (Secondary) 05B40, 46B20.}}

\author{K\'{a}roly Bezdek\thanks{Partially supported by a Natural Sciences and 
Engineering Research Council of Canada Discovery Grant.}, \ Muhammad A. Khan\thanks{Supported by Vanier CGS (NSERC) and Alberta Innovates Technology Futures.} \ and Michael Oliwa
}

\date{}

\begin{document}

\maketitle

\begin{abstract}
Contact graphs have emerged as an important tool in the study of translative packings of convex bodies. The contact number of a packing of translates of a convex body is the number of edges in the contact graph of the packing, while the Hadwiger number of a convex body is the maximum vertex degree over all such contact graphs. In this paper, we investigate the Hadwiger and contact numbers of totally separable packings of convex domains, which we refer to as the separable Hadwiger number and the separable contact number, respectively. We show that the separable Hadwiger number of any smooth convex domain is $4$ and the maximum separable contact number of any packing of $n$ translates of a smooth strictly convex domain is $\left\lfloor 2n - 2\sqrt{n}\right\rfloor$. Our proofs employ a characterization of total separability in terms of hemispherical caps on the boundary of a convex body, Auerbach bases of finite dimensional real normed spaces, angle measures in real normed planes, minimal perimeter polyominoes and an approximation of smooth $\oo$-symmetric strictly convex domains by, what we call, Auerbach domains. 
\end{abstract}

\section{Introduction}\label{sec:intro}
We denote the $d$-dimensional Euclidean space by $\E^d$ and write $\R^d$ instead when the norm is arbitrary or unknown. Let $\B^d$ denote the unit ball centered at the origin $\oo$ in $\E^d$. A {\it $d$-dimensional convex body} $\K$ is a compact convex subset of $\E^d$ with nonempty interior $\inter \K$. If $d=2$, then $\K$ is said to be a {\it convex domain}. If $\K = -\K$, where $-\K=\{-x: x\in \K\}$, then $\K$ is said to be {\it $\oo$-symmetric}. Furthermore, $\K$ is {\it centrally symmetric} if some translate of $\K$ is $\oo$-symmetric. Since the quantities we deal with here are affine invariants, we can mostly use the two terms interchangeably. Every $\oo$-symmetric convex body $\K$ in $\E^d$ induces a norm on $\R^d$, whose closed unit ball centered at $\oo$ is $\K$, given by 
\[\norm{\x}_{\K} = \inf\{\lambda>0: \x\in \lambda \K \}, 
\]
for every $\x\in \R^d$. We denote the corresponding normed space by $(\R^d , \norm{\cdot}_{\K})$. If $\K = \B^d$, then we denote the norm $\norm{\cdot}_{\B^d}$ simply by $\norm{\cdot}$. Thus $\E^d = (\R^d , \norm{\cdot})$. A $d$-dimensional convex body $\K$ is said to be {\it smooth} if at every point on the boundary $\bd \K$ of $\K$, the body $\K$ is supported by a unique hyperplane of $\E^d$ and {\it strictly convex} if the boundary of $\K$ contains no nontrivial line segment. Given $d$-dimensional convex bodies $\K$ and $\LL$, their Minkowski sum (vector sum) is denoted by $\K+\LL$ and defined as 
\[\K+\LL = \{\x+\y: \x\in \K, \y\in\LL\},
\]
which is a convex body in $\E^d$.

The {\it kissing number problem} asks for the maximum number $k(d)$ of non-overlapping translates of $\B^d$ that can touch $\B^d$ in $\E^d$. Clearly, $k(2)=6$. However, the value of $k(3)$ remained a mystery for several hundred years and caused a famous argument between Newton and Gregory in the 17th century.  To date, the only known kissing number values correspond to $d = 2, 3, 4, 8, 24$. For a survey of kissing numbers we refer the interested reader to \cite{Bo}. 

Generalizing the kissing number, the {\it Hadwiger number} or {\it the translative kissing number} $H(\mathbf{\K})$ of a $d$-dimensional convex body $\K$ is the maximum number of non-overlapping translates of $\K$ that all touch $\K$ in $\E^d$. Given the difficulty of the kissing number problem, determining Hadwiger numbers is highly nontrivial with few exact values known for $d\ge 3$. The best general upper and lower bounds on $H(\K)$ are due to Hadwiger \cite{Ha} and Talata \cite{Ta} respectively, and can be expressed as 
\begin{equation}\label{eq:hadwiger}
2^{cd}\le H(\K)\le 3^{d} - 1, 
\end{equation}
where $c$ is an absolute constant and equality holds in the right inequality if and only if $\K$ is an affine $d$-dimensional cube. For more on Hadwiger number and its relatives, we refer the reader to \cite{Be06}. 

Let $\cal{P}$ be a packing of translates of a convex body $\K$ in $\E^d$ (i.e., a family of non-overlapping translates of $\K$ in $\E^d$). The {\it contact graph} of $\cal{P}$ is the (simple) graph whose vertices correspond to the packing elements with two vertices joined by an edge if and only if the two corresponding packing elements touch each other. The number of edges of a contact graph is called the {\it contact number} of the underlying packing. The {\it contact number problem} asks for the largest number $c(\K, n, d)$ of edges in any contact graph of a packing of $n$ translates of $\K$ \cite{BeKh}. If $\K=\B^d$, we simply write $c(n, d)$ instead of $c(\B^{d}, n, d)$. The problem of determining $c(\K , n, d)$ is equivalent to the Erd\H{o}s-type repeated shortest distance problem in normed spaces proposed by Ulam \cite{Br, Er}, which asks for the largest number of repeated shortest distances among $n$ points in $(\R^{d}, \norm{\cdot}_{\K})$. Another way to look at the contact number problem is to think of it as the discrete analogue of the densest packing problem. Moreover, it has been observed by materials scientists that at low temperatures, particles of self-assembling materials such as colloids tend to form clusters so as to maximize the contact number \cite{ArMaBr, Ho}. The reader can refer to the very recent survey \cite{BeKh} on contact numbers and their applications by the first two authors for details. 

A packing of translates of a convex domain $\K$ in $\E^2$ is said to be {\it totally separable} if any two packing elements can be separated by a line of $\mathbb{E}^{2}$ disjoint from the interior of every packing element. This notion was introduced by G. Fejes T\'{o}th and L. Fejes T\'{o}th \cite{FTFT73} and has attracted significant attention. We can define a totally separable packing of translates of a $d$-dimensional convex body $\K$ in a similar way by requiring any two packing elements to be separated by a hyperplane in $\mathbb{E}^{d}$ disjoint from the interior of every packing element \cite{BeKh, BeSzSz}. One can think of a totally separable packing as a packing with barriers. Practical examples include packaging of identical fragile convex objects using cardboard separators and layered arrangements of interacting particles.  

In this paper, we introduce the totally separable analogues of the Hadwiger number and the contact number for $d$-dimensional convex bodies, which we denote by $H_{\rm sep}(\cdot)$  and $c_{\rm sep}(\cdot, n, d)$, respectively, and then mostly study them in the plane. Note that the quantity $c_{\rm sep}(n, d)=c_{\rm sep}(\B^d , n, d)$ was investigated in \cite{BeSzSz}. Some results obtained in that paper will be stated and used in the sequel. However, not much is known about the more general quantities $H_{\rm sep}(\K)$ and $c_{\rm sep}(\K, n, 2)$, where $\K$ is any planar convex domain. 

The main results of this paper can be summarized as below. 
\begin{theorem}\label{mega}
\item{(A)} $H_{\rm sep}(\K) =  4$, for any smooth convex domain $\K$ in $\E^2$.  
\item{(B)} $\csep(\K, n, 2) = \left\lfloor 2n - 2\sqrt{n}\right\rfloor$, for any smooth strictly convex domain $\K$ in $\E^2$ and $n\ge 2$. 
\end{theorem}

We prove part (A) in section \ref{sec:hadwiger} and part (B) in section \ref{sec:plane}. The tools used in proving each statement are also discussed and developed in the corresponding sections. We note that Theorem \ref{mega} (A) and (B) generalize the classical result $H_{\rm sep}(\B^2 ) = 4$ and the result $\csep(n, 2) = \left\lfloor 2n - 2\sqrt{n}\right\rfloor$, $n\ge 2$, proved by Bezdek, Szalkai and Szalkai \cite{BeSzSz}, respectively. 

We adopt the following notations. For $\x\ne \y \in \E^d$, we denote the closed (resp. open) line segment in $\E^d$ with end points $\x$ and $\y$ by $[\x, \y]$ (resp. $(\x, \y)$). Also given a centrally symmetric convex domain $\K$ and $\x\ne \y\in \bd \K$, we denote the smaller (in the norm $\norm{\cdot}_{\K}$) of the two closed (resp. open) arcs on $\bd \K$ with end points $\x$ and $\y$ by $[\x, \y]_{\K}$ (resp. $(\x, \y)_{\K}$). Ties are broken arbitrarily. All arcs considered in this paper are non-trivial, that is, different from a point.

\section{Separable Hadwiger numbers}\label{sec:hadwiger}
We begin with a lemma of Minkowski that is often used to express questions about packings of translates of an arbitrary convex body in terms of the corresponding questions for a packing of translates of an $\oo$-symmetric convex body. Given a $d$-dimensional convex body $\K$, we denote the {\it Minkowski symmetrization} of $\K$ by $\K_{\oo}$ and define it to be 
\[\K_{\oo} = \frac{1}{2}(\K + (- \K)) = \frac{1}{2}(\K - \K) = \frac{1}{2}\{\x - \y : \x , \y \in \K\}. 
\]
Clearly, $\K_{\oo}$ is an $\oo$-symmetric $d$-dimensional convex body. Minkowski \cite{Mi} showed that if $\x_{1} + \K$ and $\x_{2} + \K$ are two translates of a convex body $\K$, then they are non-overlapping (resp., touching) if and only if $\x_{1} + \K_{\oo}$ and $\x_{2} + \K_{\oo}$ are non-overlapping (resp., touching). Thus if $\K$ is a convex body and ${\cal{P}}=\{\x_1+\K , \ldots, \x_n + \K\}$ is a packing of translates of $\K$ in $\E^d$, then ${\cal{P}}_{\oo}=\{\x_1+\K_{\oo} , \ldots, \x_n + \K_{\oo}\}$ is a packing and vice versa. Moreover, the contact graphs of $\cal{P}$ and ${\cal{P}}_{\oo}$ are identical. Here we prove the following additional property of Minkowski symmetrization.

\begin{figure}[t]
\centering
\includegraphics[scale=1.1]{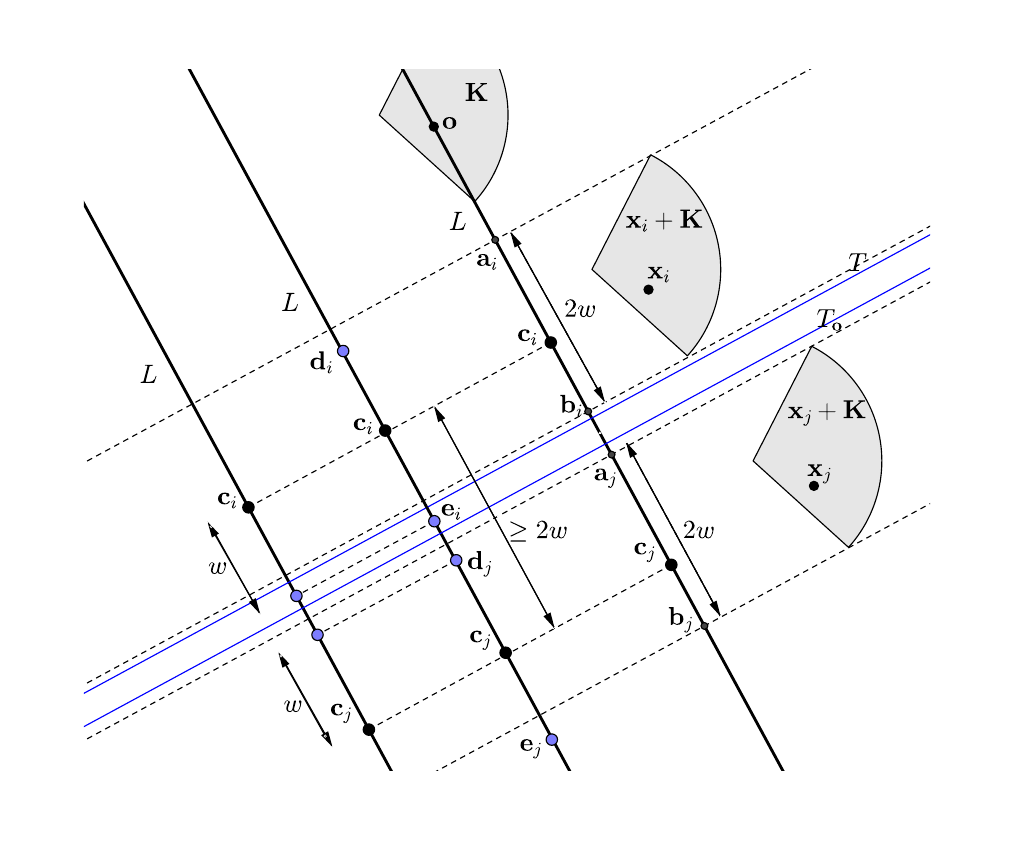}
\vspace{-10mm}
\caption{Lemma \ref{minkowski} -- the totally separable form of Minkowski's lemma. }
\label{fig:mink}
\end{figure}

\begin{lemma}\label{minkowski}
Let $\K$ be a convex body and ${\cal{P}} = \{\x_1 + \K , \ldots, \x_n + \K\}$ be a set of translates of $\K$ in $\E^d$. Then $\cal{P}$ forms a totally separable packing if and only if ${\cal{P}}_{\oo} = \{\x_1 + \K_{\oo} , \ldots, \x_n + \K_{\oo}\}$ is a totally separable packing of translates of $\K_{\oo}$. 
\end{lemma}

\begin{proof}
Clearly, $\x_i + \K_{\oo}$ is a centrally symmetric convex body with center $\x_i$, for all $1\le i\le n$. For simplicity, let $\oo\in \inter \K$. Then $\x_i \in \x_i +\inter \K$, for all $1\le i\le n$. 

Assume that the packing $\cal{P}$ is totally separable. Then for any distinct $\x_p +\K , \x_q +\K \in \cal{P}$, there exists a hyperplane $T$ of $\E^d$ that separates $\x_p +\K$ and $\x_q +\K$ and is disjoint from the interiors of all packing elements in ${\cal{P}}$. Thus $T$ partitions $\cal{P}$ into two subsets, each containing the elements of $\cal{P}$ that lie in one closed half-space of $\E^d$ bounded by $T$. Arbitrarily call one of these closed half-spaces the {\it left} of $T$, while the other the {\it right} of $T$. Let $\x_i +\K \in \cal{P}$ be to the left of $T$ that is closest to $T$ (in the norm $\norm{\cdot}$) and $\x_j + \K \in \cal{P}$ be to the right of $T$ that is closest to $T$, breaking ties arbitrarily. Let $L$ be the line orthogonal to $T$ and passing through $\oo$. We project the packings $\cal{P}$ and $\cal{P}_{\oo}$ on to $L$ using the orthogonal projection map $\pi:\E^d \to L$.   

We refer to Figure \ref{fig:mink}. Let $[\aaa_i , \bb_i ] = \pi(\x_i + \K )$ and $[\aaa_j ,\bb_j ] =\pi(\x_j + \K )$. Also let $\cc_i =\pi(\x_i )$ and $\cc_j = \pi(\x_j )$. Finally, if $\mathbf{d}_i = \mathbf{c}_i - \frac{1}{2}(\mathbf{b}_i - \mathbf{a}_i )$, $\mathbf{d}_j = \mathbf{c}_j - \frac{1}{2}(\mathbf{b}_j - \mathbf{a}_j )$, $\mathbf{e}_i = \mathbf{c}_i + \frac{1}{2}(\mathbf{b}_i - \mathbf{a}_i )$, $\mathbf{e}_j = \mathbf{c}_j + \frac{1}{2}(\mathbf{b}_j - \mathbf{a}_j )$, then $[\dd_i , \ee_i ] = \pi(\x_i + \K_{\oo})$ and $[\dd_j , \ee_j ] = \pi(\x_j + \K_{\oo})$. Let $\norm{\bb_i - \aaa_i } = \norm{\bb_j - \aaa_j } = 2w$. Then $\norm{\cc_j - \cc_i } \ge 2w$ and therefore, the closed line segments $[\mathbf{d}_i , \mathbf{e}_i ]$ and $[\mathbf{d}_j , \mathbf{e}_j ]$ do not overlap. Thus there exists a translate $T_{\oo}$ of $T$ that separates $\x_i +\K_{\oo}$ and $\x_j +\K_{\oo}$ and is disjoint from the interiors of all packing elements in ${\cal{P}}_{\oo}$. Finally, $T_{\oo}$ separates $\x_p +\K_{\oo}$ and $\x_q + \K_{\oo}$.  

Conversely, assume that the packing $\cal{P}_{\oo}$ is totally separable. Then there exists a hyperplane $T_{\oo}$ of $\E^d$ that separates $\x_p + \K_{\oo}$ and $\x_q +\K_{\oo}$ and is disjoint from the interiors of all packing elements in ${\cal{P}}_{\oo}$. Arguing on similar lines as above, let $\x_i +\K_{\oo}$ be an element of $\cal{P}_{\oo}$ to the left of $T_{\oo}$ that is closest to $T_{\oo}$ and $\x_j + \K_{\oo}$ be an element of $\cal{P}_{\oo}$ to the right of $T_{\oo}$ that is closest to $T_{\oo}$. Then by arguing as above, we can easily show that there is a translate $T$ of $T_{\oo}$ that separates $\x_i + \K$ and $\x_j + \K$ and is disjoint from the interiors of all packing elements in ${\cal{P}}$. Moreover, it also separates $\x_p + \K$ and $\x_q + \K$. 
\end{proof}

We now define a totally separable analogue of the Hadwiger number, which we call the {\it separable Hadwiger number}. 

\begin{definition}
Let $\K$ be a convex body in $\E^d$. We define the separable Hadwiger number $H_{\rm sep}(\K)$ of $\K$ as the maximum number of translates of $\K$ that all touch $\K$ and, together with $\K$, form a totally separable packing in $\E^d$. 
\end{definition} 

Although the above definition is very natural, it is not very helpful in determining the exact values or good estimates of $H_{\rm sep}(\K)$. For smooth $\oo$-symmetric convex bodies we will find it advantageous to use an alternative but equivalent characterization. We begin by defining what we mean by a {\it hemispherical cap} on the boundary of a smooth $\oo$-symmetric convex body.  

\begin{definition}\label{def:caps}
Let $\K_{\oo}$ be a smooth $\oo$-symmetric convex body in $\E^d$. Let $\p\in \bd \K_{\oo}$ and $T_{\p}$ denote the unique supporting hyperplane of $\K_{\oo}$ at $\p$. Let $T'_{\p}$ be the hyperplane parallel to $T_{\p}$ passing through the origin and let $P_{\p}$ be the half-open plank of $\E^d$ bounded by $T_{\p}$ and $T'_{\p}$, but excluding $T'_{\p}$. We say that $C(\p):=\bd \K_{\oo} \cap P_{\p}$ is the {\it hemispherical cap (or simply the cap) on $\bd \K_{\oo}$ centered at $\p$}. We define the {\it boundary of the cap $C(\p)$} as $\bd\K_{\oo} \cap T'_{\p}$ and denote it by $\partial C(\p)$. 
\end{definition}

\begin{figure}[t]
\centering
\includegraphics[scale=0.4]{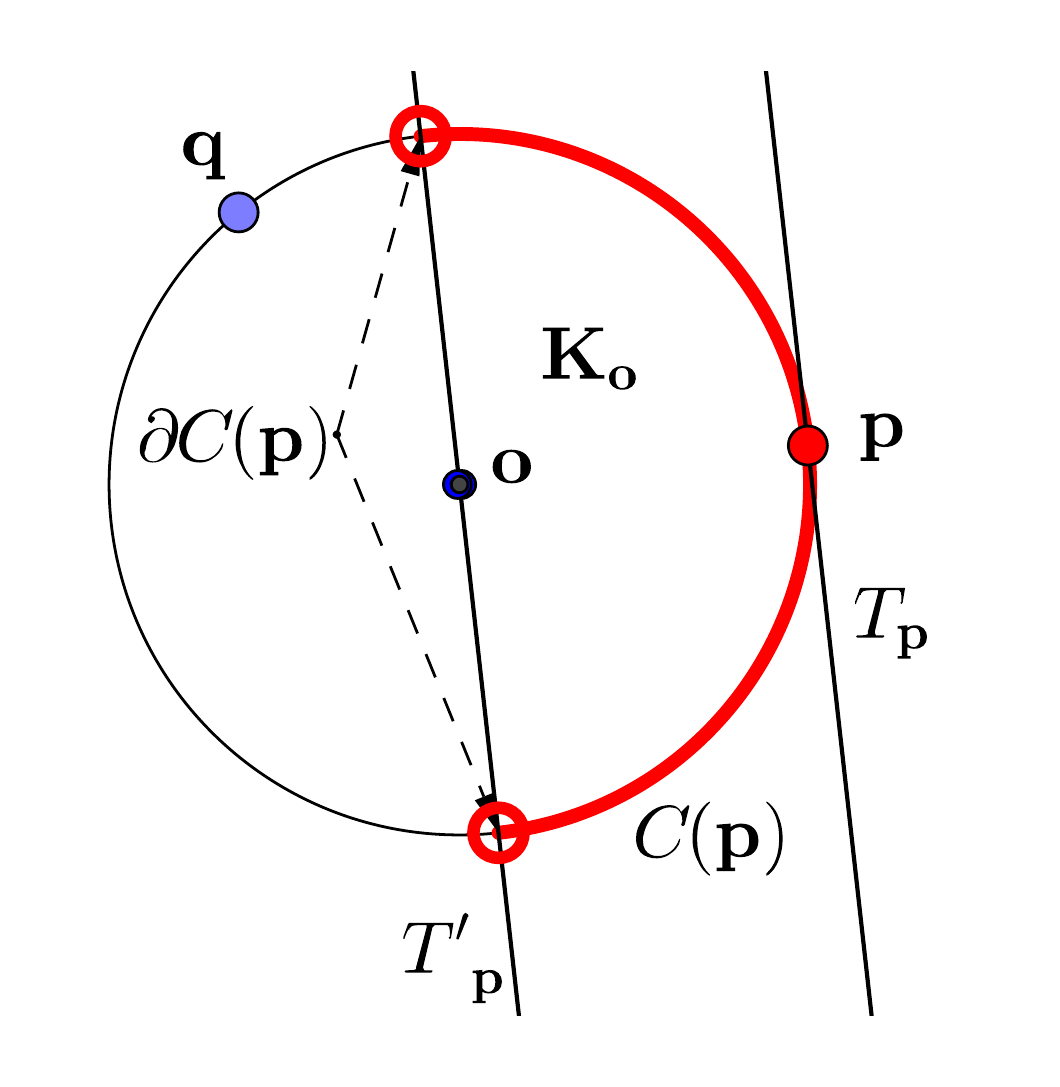}
\vspace{-5mm}
\caption{The cap $C(\p)$ on a convex domain $\K_{\oo}$ (a circular disk in this case) centered at $\p$, its boundary $\partial C(\p)$ and a separable point set $\{\p, \q\}$ on $\bd \K_{\oo}$.}
\label{fig:caps}
\end{figure}

Figure \ref{fig:caps} illustrates Definition \ref{def:caps} for a planar convex domain. The reason why we are defining caps only for smooth convex bodies is obvious. We wish to avoid situations where multiple hyperplanes support $\K_{\oo}$ at a single boundary point. Also $\oo$-symmetry is used in defining the caps. We now introduce the notion of a {\it separable point set} on the boundary of a smooth $\oo$-symmetric convex body $\K_{\oo}$. 

\begin{definition}
Let $\K_{\oo}$ be a smooth $\oo$-symmetric convex body in $\E^d$. A set $S\subseteq \bd \K_{\oo}$ is called a {\it separable point set} if $\x\notin C(\y)$ and $\y\notin C(\x)$, for any $\x, \y\in S$. 
\end{definition}

Next, we characterize total separability of translative packings of a smooth convex body $\K_{\oo}$ in terms of separable point sets defined on $\bd \K_{\oo}$. 

\begin{lemma}\label{caps}
Let $\K_{\oo}$ be a smooth $\oo$-symmetric convex body in $\E^d$. Let $\p_1 , \p_2 \in \bd \K_{\oo}$ and let $\K_1 = 2\p_1 + \K_{\oo}$ and $\K_2 = 2\p_2 + \K_{\oo}$. Then $\K_{\oo}$, $\K_1$ and $\K_2$ form a totally separable packing if and only if $\p_{1}\notin C(\p_{2})$ and $\p_{2}\notin C(\p_{1})$. 
\end{lemma}

\begin{proof}
Let $T_{1}$ be the unique hyperplane in $\E^d$ supporting $\K_{\oo}$ (and $\K_1$) at the point $\p_{1}$ and let $T'_{1}$ be the hyperplane parallel to $T_{1}$ passing through $\oo$. Note that $T_{1}$ is also the unique hyperplane in $\E^d$ separating $\K_{\oo}$ and $\K_1$. Similarly, let $T_{2}$ be the unique hyperplane supporting $\K_{\oo}$ (and $\K_2$) at $\p_2$ and $T'_{2}$ be the parallel hyperplane passing through $\oo$. For any hyperplane $T$ of $\E^d$, we write $T^+$ and $T^-$ to denote the two closed half-spaces of $\E^d$ bounded by $T$. Let $\K_{i}\subseteq T^+_{i}$, $i=1, 2$.  

($\Leftarrow$) If $\p_{2}\notin C(\p_{1})$, then it is easy to observe that $\K_2 \subseteq T^-_1$. Thus $T_{1}$ separates $\K_1$ from both $\K_{2}$ and $\K_{\oo}$. Similarly, if $\p_{1}\notin C(\p_{2})$, then $\K_1 \subseteq T^-_2$ and $T_{2}$ separates $\K_2$ from both $\K_{1}$ and $\K_{\oo}$.  

($\Rightarrow$) Now suppose that $\K_{\oo}$, $\K_1$ and $\K_2$ form a totally separable packing. Let us also assume that $\K_2 \subseteq T^+_1$. By the total separability of packing $\{\K_{\oo}, \K_1 , \K_2 \}$, we must have $\p_2 \in T_1$ and so, $T_1$ does not separate $\K_1$ and $\K_2$. If $\p_1 =\p_2$, then any hyperplane $L\ne T_1$ separating $\K_{1}$ and $\K_2$ must pass through $\p_1$. Again by the total separability assumption, $L\cap \inter \K_{\oo}=\varnothing$. Therefore, $L$ is a supporting hyperplane of $\K_{\oo}$ at $\p_1$, which contradicts the smoothness of $\K_{\oo}$. Thus $\p_1 \ne \p_2$ and the (non-trivial) line segment $[\p_1 , \p_2 ]$ lies on $T_1 \cap \bd\K_{\oo}$ as shown in Figure \ref{fig:new}. By the $\oo$-symmetry of $\K_{\oo}$, $[-\p_1 , -\p_2]\in \bd\K_{\oo}$ and the parallelogram $\mathbf{P}$ with vertices $\p_1$, $\p_2$, $-\p_1$ and $-\p_2$ satisfies ${\rm relint}\mathbf{P}\subseteq \inter \K_{\oo}$. Also by the symmetry of packing $\{\K_{\oo}, \K_1 \}$ about $\p_1$ and of $\{\K_{\oo}, \K_2 \}$ about $\p_2$, we have $\mathbf{P}_1 = 2\p_1 +\mathbf{P}\subseteq \K_1$ and $\mathbf{P}_2 = 2\p_2 + \mathbf{P}\subseteq \K_2$. Let $\E^2 := {\rm span}\{\p_1 , \p_2 \}$. Therefore, for any hyperplane $L\ne T_1$ separating $\K_1$ and $\K_2$, the line $L\cap \E^2$ separates $\mathbf{P}_1$ and $\mathbf{P}_2$ in $\E^2$. Since $L$ cannot intersect ${\rm relint} \mathbf{P}$, it either passes through $\p_1$ or $\p_2$ (but not both). This implies that $L$ supports $\K_{\oo}$ at $\p_1$ or at $\p_2$, a contradiction. Thus $\K_2 \subseteq T^-_1$ and clearly, $\p_2 \notin C(\p_1 )$. Similarly, it can be shown that $\p_1 \notin C(\p_2 )$. 
\end{proof}

\begin{figure}[t]
\centering
\includegraphics[scale=.8]{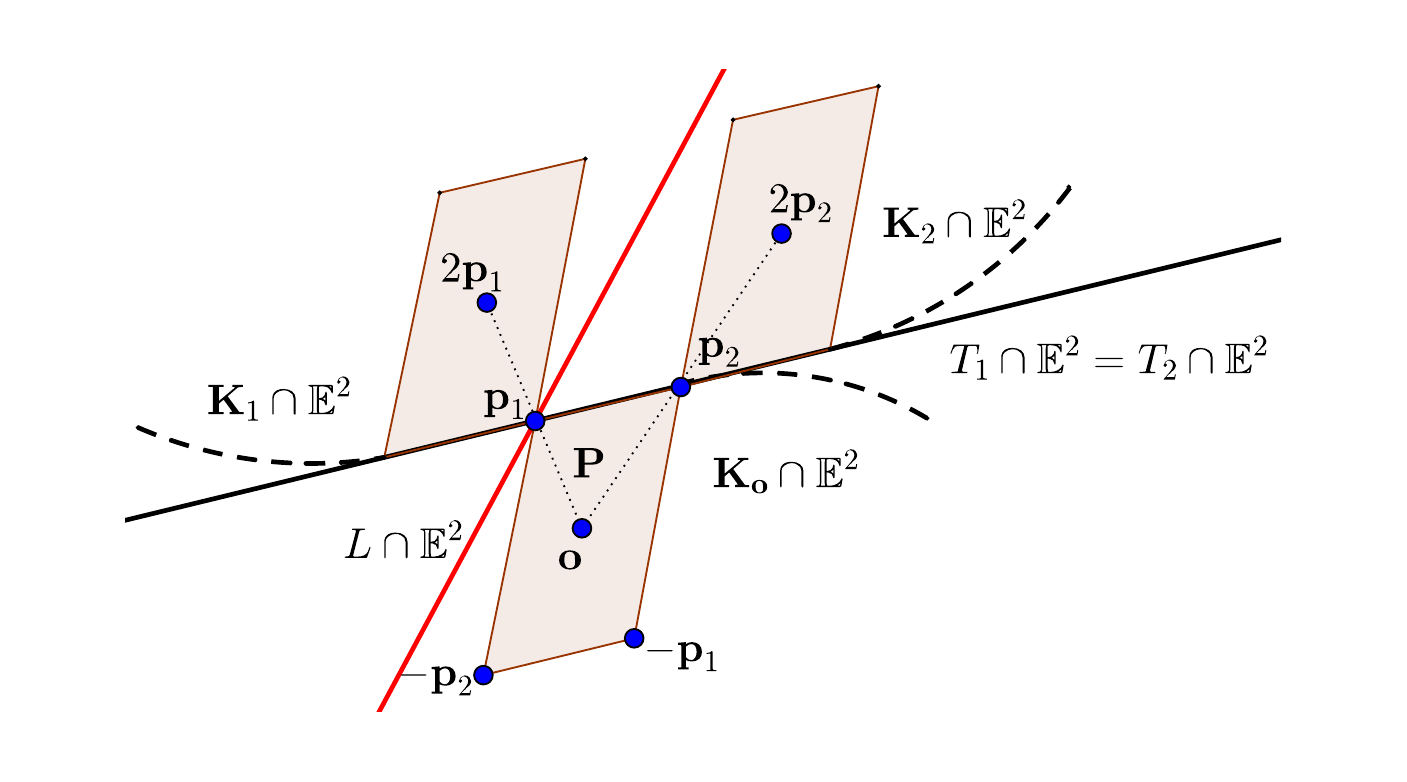}
\vspace{-10mm}
\caption{The contradiction in Lemma \ref{caps} arising from the assumption $\K_2 \subseteq T^+_1$.}
\label{fig:new}
\end{figure}

The following is an immediate consequence of Lemma \ref{caps} and the fact that in any totally separable packing involving $\K_{\oo}$, $\K_1$, $\K_2$ and possibly other translates of $\K_{\oo}$ all touching $\K_{\oo}$ as defined above, $T_{1 }$ is the unique hyperplane of $\E^d$ separating $\K_1$ from all other packing elements. 

\begin{corollary}\label{alternative}
If $\K_{\oo}$ is a smooth $\oo$-symmetric convex body in $\E^d$, then $H_{\rm sep}(\K_{\oo})$ equals the maximum cardinality of a separable point set on $\bd \K_{\oo}$. 
\end{corollary}

\begin{definition}
Let $\K_{\oo}$ be an $\oo$-symmetric convex body in $\E^d$, $d\ge 2$. A non-zero vector $\x$ in $(\R^{d} , \norm{\cdot}_{\K_{\oo}} )$ is said to be {\em Birkhoff orthogonal} to a non-zero vector $\y$ if $\norm{\x}_{\K_{\oo}}\leq \norm{\x+t\y}_{\K_{\oo}}$, for all $t\in \mathbb{R}$ \cite{Bir}. We denote this by $\x\dashv_{\K_{\oo}}\y$. 
\end{definition}

Note that in general, Birkhoff orthogonality is a  non-symmetric relation, that is $\x\dashv_{\K_{\oo}}\y$ does not imply $\y\dashv_{\K_{\oo}}\x$. 

\begin{definition}
Let $\K_{\oo}$ be an $\oo$-symmetric convex body in $\E^d$, $d\ge 2$, and $\{\e_{i}:i=1,\ldots, d\}$ be a basis for $\R^d$. We call this an {\em Auerbach basis} of $(\R^{d} , \norm{\cdot}_{\K_{\oo}} )$ provided that for every $i$, $\norm{\e_{i}}_{\K_{\oo}}=1$ and $\e_{i}$ is Birkhoff orthogonal to every element of the linear subspace $\mathrm{span}\{\e_{j}:j\neq i, j=1,\ldots, d\}$ (see \cite{Pl}). 
\end{definition}

Plichko \cite{Pl} showed that for every $\oo$-symmetric $d$-dimensional convex body $\K_{\oo}$, the normed linear space $(\R^{d} , \norm{\cdot}_{\K_{\oo}} )$ possesses at least two Auerbach bases -- one corresponding to the centers of the facets of the affine $d$-cube of minimum volume circumscribing $\K_{\oo}$, while the other corresponding to the vertices of the affine $d$-cross polytope of maximum volume inscribed in $\K_{\oo}$. Moreover, if these two bases coincide, then $(\R^{d} , \norm{\cdot}_{\K_{\oo}} )$ possesses infinitely many Auerbach bases.

\begin{remark}\label{rem}
Suppose $\K_{\oo}$ is a smooth $\oo$-symmetric convex body in $\E^d$ and $\x, \y\in \bd\K_{\oo}$ with $\x\dashv_{\K_{\oo}}\y$. If $T$ is the hyperplane supporting $\K_{\oo}$ at $\x$ and $T'$ is the hyperplane passing through $\oo$ and parallel to $T$, then clearly $\y\in T'$. Conversely, if $\x, \y\in \bd\K_{\oo}$ and $\y\in T'$, then $\x\dashv_{\K_{\oo}}\y$. It follows that $\x\dashv_{\K_{\oo}}\y$ if and only if $\y\in \partial C(\x)$. 
\end{remark}

\begin{lemma}\label{lower}
If $\K$ is a convex body in $\E^d$, $d\ge 2$, then 
\begin{equation}\label{eq:lower}
H_{\rm sep}(\K)\geq 2d.
\end{equation}
\end{lemma}

\begin{proof}
Indeed by Lemma \ref{minkowski}, $H_{\rm sep}(\K) = H_{\rm sep}(\K_{\oo})$. Furthermore, if $\{\e_{i}:i=1,\ldots, d\}$ is an Auerbach basis of $(\R^{d} , \norm{\cdot}_{\K_{\oo}} )$, then $S=\{\pm\e_{i}:i=1,\ldots, d\}$ is a separable point set on $\bd \K_{\oo}$. Thus using \cite{Pl} we obtain that $H_{\rm sep}(\mathbf{K}_{\oo}) \geq \left|S\right| = 2d$. 
\end{proof}

We now prove that for $d=2$, the lower bound given in Lemma \ref{lower} becomes exact. 

\begin{figure}[t]
\centering
\includegraphics[scale=.8]{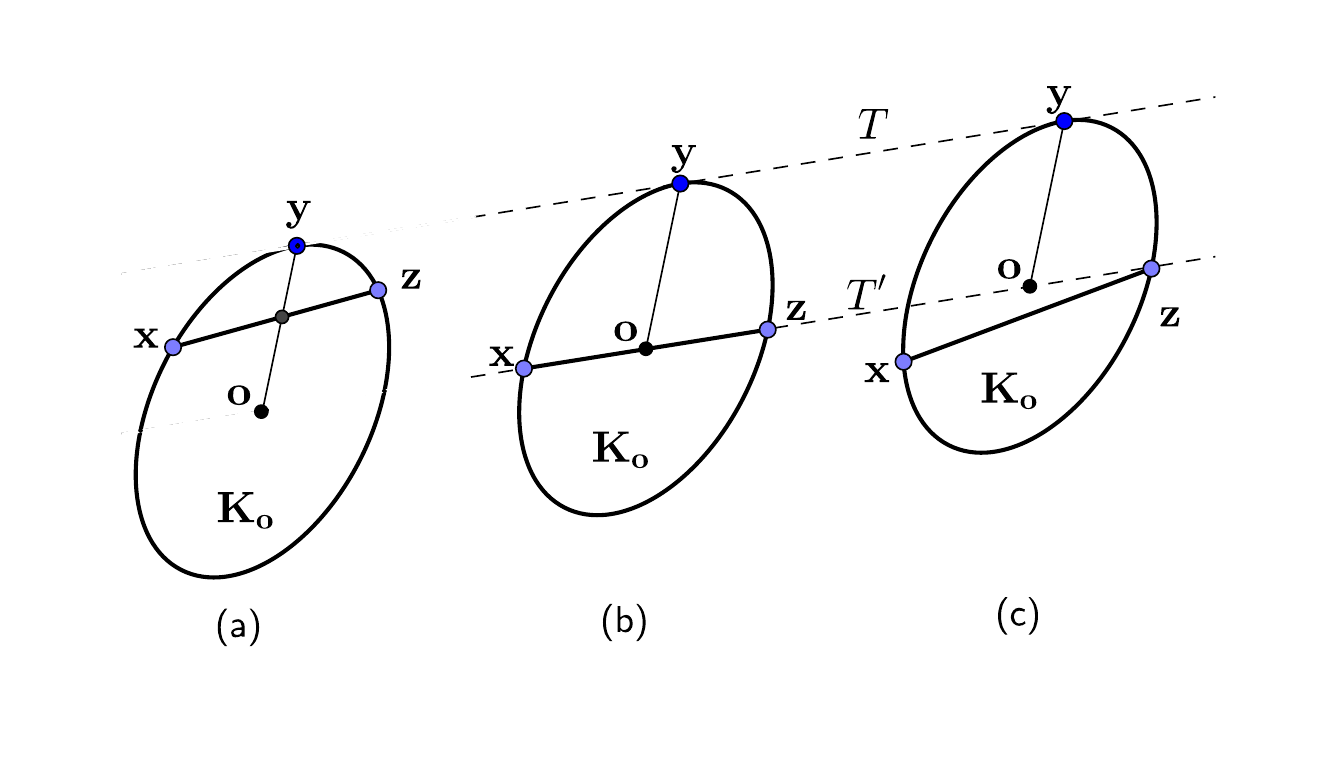}

\vspace{-12mm}

\caption{In Lemma \ref{hull}, (a) three points in a separable point set on $\bd\K_{\oo}$ and the situation with (b) $\x, \z\in\partial C(\y)$ or (c) $\x\notin\partial C(\y)$.}
\label{fig:hull}
\end{figure}

\begin{lemma}\label{hull}
Let $\K_{\oo}$ be a smooth $\oo$-symmetric convex domain in $\E^2$ and $S$ be any separable point set consisting of $3$ or more points on $\bd \K_{\oo}$. Then $\oo \in \conv(S)$, where $\conv(S)$ stands for the convex hull of $S$. 
\end{lemma}

\begin{proof}
Suppose to the contrary that this is not the case. Then there exists a separable point set $S$ on $\bd \K_{\oo}$ with $\x , \y , \z \in S$ and $\oo \notin \conv\{\x , \y , \z \}$. Orienting the boundary of $\K_{\oo}$ in an arbitrary direction, say counterclockwise, gives an ordering of these three points, say $\x , \y , \z$, such that the line segment $[\x, \z]$ intersects the line segment $[\oo, \y]$ as shown in Figure \ref{fig:hull}. (Note that the point of intersection could be any point on the half-open line segment $(\oo,\y]$.) Let $T$ be the unique line supporting $\K_{\oo}$ at $\y$ and $T'$ the line parallel to $T$ passing through $\oo$. By assumption, $\x, \z \notin C(\y)$. If $\x, \z\in \partial C(\y)$, then $\oo\in \conv\{\x, \z\}$, a contradiction. If either $\x\notin \partial C(\y)$ or $\z\notin \partial C(\y)$, then $[\x, \z]$ is disjoint from $[\oo,\y]$, again a contradiction. 
\end{proof}

We can now compute the separable Hadwiger number of smooth $\oo$-symmetric convex domains exactly.  

\begin{figure}[t]
\centering
\includegraphics[scale=0.4]{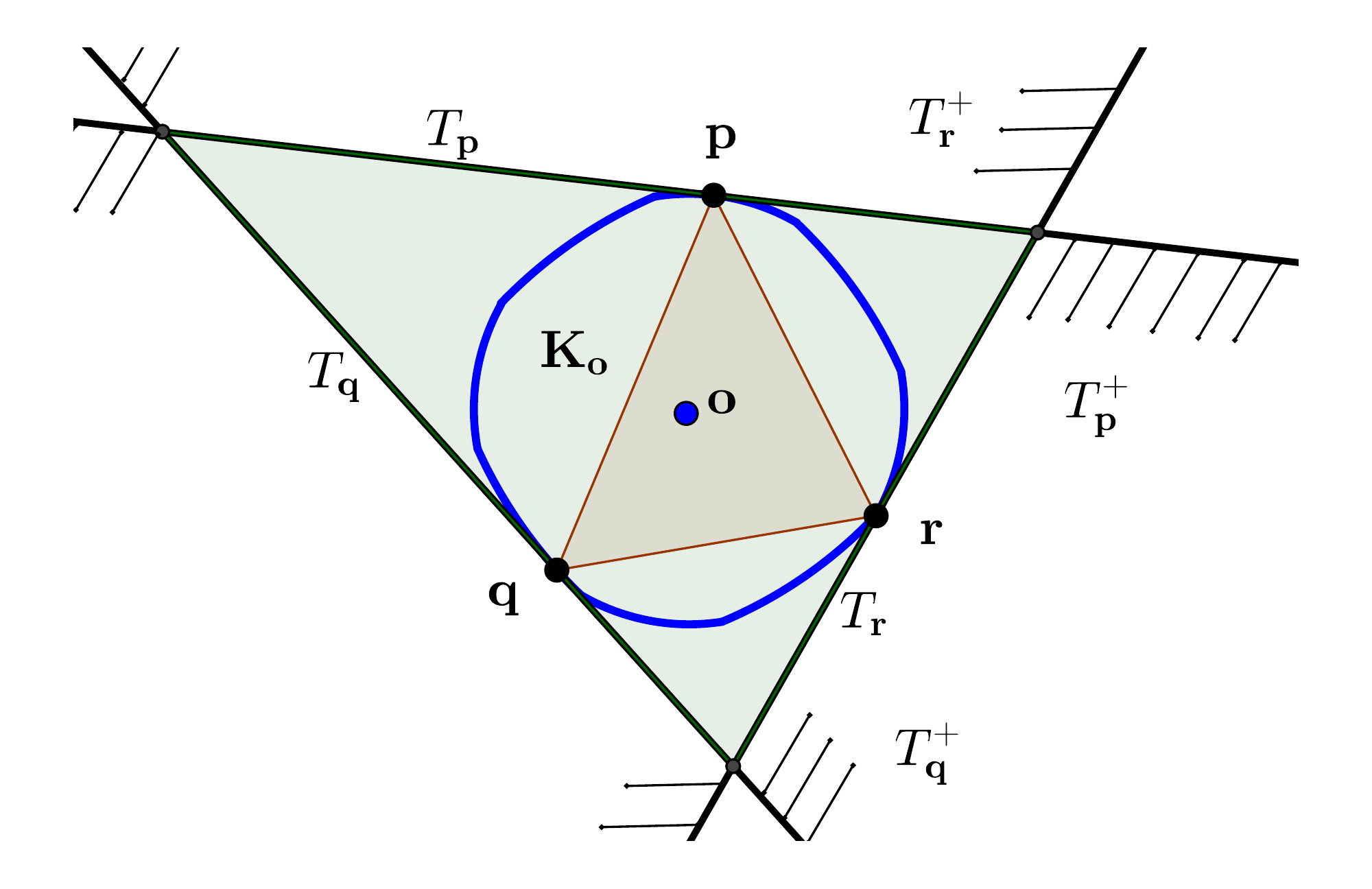}
\vspace{-5mm}
\caption{The case $k=2$ of Theorem \ref{2Dmain} when no two of the vertices of $\Delta^2$ have parallel tangent lines supporting $\K_{\oo}$.}
\label{fig:triangle}
\end{figure}

\begin{theorem}\label{2Dmain}
If $\K_{\oo}$ is any smooth $\oo$-symmetric convex domain in $\E^2$, then 
\begin{equation}\label{eq:main}
H_{\rm sep}(\K_{\oo})=4. 
\end{equation}
Moreover, a separable point set $S$ on $\bd \K_{\oo}$ has maximal cardinality if and only if elements of $S$, when considered as vectors, form an Auerbach basis of $(\mathbb{R}^2, \norm{\cdot}_{\K_{\oo}})$. 
\end{theorem}

\begin{proof}
First, we prove \eqref{eq:main} using Corollary \ref{alternative}. Suppose $S$ is a maximal cardinality separable point set on $\bd \K_{\oo}$. By Lemma \ref{hull}, $\oo\in \conv(S)$ and by the Carath\'{e}odory theorem, there exists a simplex $\Delta^k$ of minimum dimension $1\le k\le 2$, such that $\oo \in {\rm relint} \Delta^k$. Two cases arise. 

If $k=1$, then $\oo$ lies on a line segment joining a pair of antipodal points, say $\p$ and $-\p$, in $S$. Since $\K_{\oo}$ is smooth and $\oo$-symmetric, the unique lines tangent to $\K_{\oo}$ at $\p$ and $-\p$ are parallel. Hence $\K_{\oo}\setminus (C(\p)\cup C(-\p))$ consists of a pair of antipodal points and so, $\left| S\right| \le 4$. The result follows by Lemma \ref{lower}. 

If $k=2$, then we deal with two subcases. Let $V(\Delta^2 ) =\{\p, \q, \rr\}$ be the set of vertices of $\Delta^2$.  Suppose that the tangent lines to two of the vertices of $\Delta^2$, say $\p$ and $\q$, are parallel. Then once again $\K_{\oo}\setminus (C(\p)\cup C(\q))$ consists of a pair of antipodal points and the result follows. 

Now suppose that no two of the vertices of $\Delta^2$ are supported by parallel lines (Figure \ref{fig:triangle}). Let $T_{\p}$, $T_{\q}$ and $T_{\rr}$ be the unique lines tangent to $\K_{\oo}$ at $\p$, $\q$ and $\rr$ bounding the closed half-spaces $T_{\p}^+$, $T_{\p}^-$, $T_{\q}^+$, $T_{\q}^-$, $T_{\rr}^+$ and $T_{\rr}^-$ of $\E^d$, respectively, such that $\K_{\oo}\subseteq T_{\p}^+ \cap T_{\q}^+ \cap T_{\rr}^+$ and $(-\p+T_{\p}^+ )\cap (-\q+T_{\q}^+ )\cap (-\rr+T_{\rr}^+ )= \{\oo\}$. Thus $T_{\p}$, $T_{\q}$ and $T_{\rr}$ form a triangle circumscribing $\K_{\oo}$ and 
$$\bd \K_{\oo} = C(\p) \cup C(\q)\cup C(\rr), $$ 
showing that $\left|S\right| = \left|V(\Delta^2 )\right| = 3$. From Lemma \ref{lower}, this contradicts the maximality of $\left|S\right|$. 

The proof of \eqref{eq:main} also shows that a maximal cardinality separable point set on $\bd\K_{\oo}$ consists of a pair of antipodes $\{\p, -\p, \q, -\q\}$ such that $\p\in \partial C(\q)$, $\q\in \partial C(\p)$ and $\p$ and $\q$, when considered as vectors, are linearly independent. Thus by Remark \ref{rem}, $\p\dashv_{\K_{\oo}}\q$ and $\q\dashv_{\K_{\oo}} \p$ and therefore, $\{\p, \q\}$ is an Auerbach basis of $(\mathbb{R}^2, \norm{\cdot}_{\K_{\oo}})$. Conversely, if $\{\p, \q\}$ is an Auerbach basis of $(\mathbb{R}^2, \norm{\cdot}_{\K_{\oo}})$, then by Remark \ref{rem}, $\{\p, -\p, \q, -\q\}$ is a separable point set on $\bd\K_{\oo}$, and by \eqref{eq:main}, has maximal cardinality. This completes the proof.  
\end{proof}


We now refer to Lemma \ref{minkowski}. Despite using $\oo$-symmetry throughout most of this section, Lemma \ref{minkowski} shows that this assumption can be removed from the statement of Theorem \ref{2Dmain}.

\begin{corollary}\label{nonsym}
If $\K$ is a smooth convex domain in $\E^2$, then $H_{\rm sep}(\K)=4$. 
\end{corollary}


\begin{figure}[t]
\centering
\includegraphics[width=6cm, height=3.5cm]{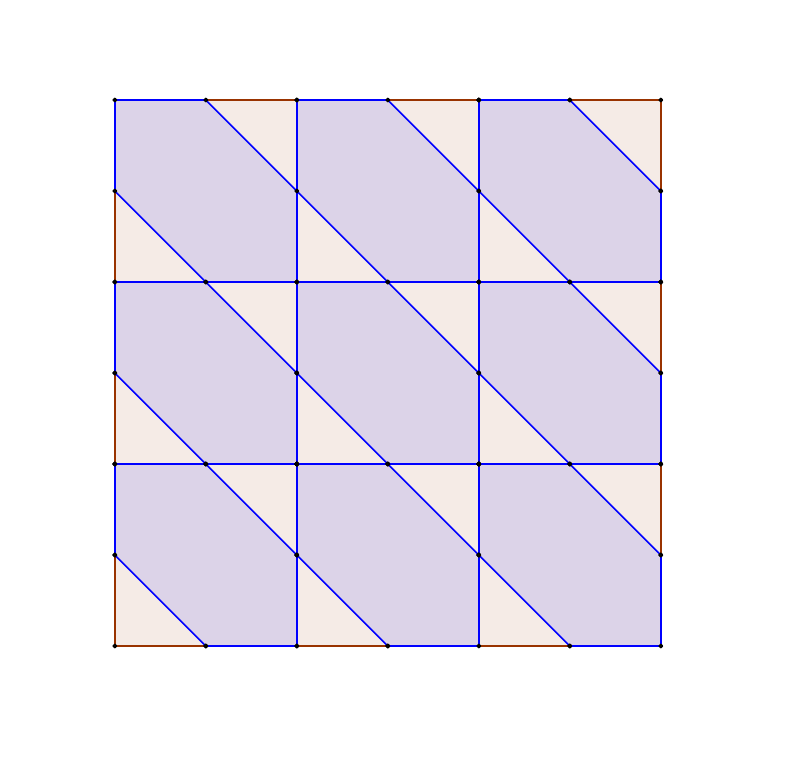}
\vspace{-7mm}
\caption{Totally separable lattice packings of affine squares and inscribed affine regular convex hexagons demonstrating a Hadwiger number of $8$ and $6$, respectively.}
\label{fig:lattice}
\end{figure}

This proves part (A) of Theorem \ref{mega}. Note that the result is sharp in the sense that there exist non-smooth convex domains with higher separable Hadwiger numbers as shown in Figure \ref{fig:lattice}. In fact, using a result of Gr\"{u}nbaum \cite{Gr}, it follows that for any convex domain $\K$, $H_{\rm sep}(\K)\in \{4, 5, 6, 8\}$.

\section{Maximum separable contact numbers} \label{sec:plane}
In the classical paper \cite{Har}, Harborth showed that in the Euclidean plane 
\begin{equation}\label{eq:harborth}
c(n, 2)=\lfloor 3n-\sqrt{12n-3}\rfloor, 
\end{equation}
for all $n\ge 2$. Brass \cite{Br} extended \eqref{eq:harborth} showing that if ${\bf K}$ is a convex domain different from an affine square in $\mathbb{E}^{2}$ then for all $n\geq 2$, one has $c({\bf K},n,2)=\lfloor 3n-\sqrt{12n-3}\rfloor$ and if ${\bf K}$ is an affine square, then $c({\bf K},n,2)=\lfloor 4n-\sqrt{28n-12}\rfloor$ holds for all $n\ge 2$. 

\begin{definition}
Let $\K$ be a convex body in $\E^d$ and $n$ be a positive integer. We define the {\it maximum separable contact number} $\csep(\K, n, d)$ of $\K$ as the maximum number of edges in a contact graph of a totally separable packing of $n$ translates of $\K$ in $\E^d$.
\end{definition}

The following natural question was raised in \cite{BeKh}.

\begin{problem}\label{separable brass}
Find an analogue of Brass' result for totally separable translative packings of an arbitrary convex domain $\K$. 
\end{problem}

Recently, Bezdek, Szalkai  and Szalkai \cite{BeSzSz} used Harborth's method to solve the Euclidean case of Problem \ref{separable brass}, namely, they proved that 
\begin{equation}\label{eq:bezdek}
\csep (n,2) = \floor{2n-2\sqrt{n}}. 
\end{equation}
In this section, we address Problem \ref{separable brass} for every smooth strictly convex domain $\K$, showing that $c_{\rm sep}(\K, n, 2) = c_{\rm sep}(n, 2)$, which proves part (B) of Theorem \ref{mega}. The proof technique, among other ideas, makes use of angle measures in normed planes, which we discuss now.   

\begin{definition}\label{def:angle}
Let $\K_{\oo}\subseteq \R^2$ be an $\oo$-symmetric convex domain in $\E^2$. An {\it angular measure}, also called an {\it angle measure}, in $(\R^{2}, \norm{\cdot}_{\K_{\oo}})$ is a measure $\mu$ defined on $\bd \K_{\oo}$ that can be extended in a translation-invariant way to measure angles anywhere and satisfies the following properties \cite{Br}: 
\begin{enumerate}
    \item[(i)] $\mu(\bd \K_{\oo})=2\pi$. 
    \item[(ii)] For any Borel set $X\subseteq \bd \K_{\oo}$, $\mu(X)=\mu(-X)$. 
    \item[(iii)] For each $\x\in \bd \K_{\oo}$, $\mu(\{\x\})=0$. 
\end{enumerate}
\end{definition}

For any $\x, \y\in \bd \K_{\oo}$, we write $\mu([\x,\y]_{\K_{\oo}})$ for the measure of the angle subtended by the arc $[\x, \y]_{\K_{\oo}}$ at $\oo$. In \cite{BaHoMaTe, Du}, angle measures are required to satisfy a fourth non-degeneracy condition, namely, for any $\x\ne \y \in \bd \K_{\oo}$, $\mu([\x, \y]_{\K_{\oo}})>0$. Here it suffices to adopt Brass' definition. We refer the interested reader to \cite{BaHoMaTe} for a very recent expository treatment of angle measures. 

Note that the usual Euclidean angle measure in the plane satisfies these conditions. Moreover for any angle measure in $(\R^{2}, \norm{\cdot}_{\K_{\oo}})$, the sum of interior angles of any simple $n$-gon in $\R^{2}$ equals $(n-2)\pi$ \cite{Br}. This observation and the following notion of a $B$-measure will be used in the proof of Theorem \ref{circular}. 

\begin{definition}
An angular measure $\mu$ in the plane $(\R^{2}, \norm{\cdot}_{\K_{\oo}})$ is called a {\em $B$-measure} \cite{Fa} if for any $\x,\y\in \bd \K_{\oo}$, $\x\dashv_{\K_{\oo}}\y$ implies that $\mu([\x,\y]_{\K_{\oo}})=\pi/2$.
\end{definition}




\subsection{Smooth $B$-domains and their maximum separable contact numbers}\label{sec:circular}
In this section, we define a class of convex domains, which we call $B$-domains, and obtain an exact formula for the maximum contact number of totally separable packings of $n$ translates of any smooth $B$-domain. The name $B$-domain stems from the connection with $B$-measures.


\begin{definition}\label{b-domain}
Let $\mathbf{D}\subseteq \E^2$ be an $\oo$-symmetric convex domain, then $\mathbf{D}$ is called a {\it $B$-domain} if there is a $B$-measure defined in $(\R^2 , \norm{\cdot}_{\mathbf{D}})$.
\end{definition}

\noindent From \cite{Fa}, the class of $B$-domains includes circular disk, unit disks of Radon planes including affine regular convex hexagon (in fact, all regular convex $2n$-gons, where $n$ is odd \cite{MaSwWe}) and, more generally, convex domains whose boundary contains a Radon arc. We will shortly see how having a $B$-measure helps when computing the maximum separable contact number. 

Before stating the main result of this section, we take a detour to introduce some ideas that will be needed in its proof. Consider the two-dimensional integer lattice ${\mathbb{Z}}^{2}$, which can also be thought of as an infinite plane tiling array of unit squares called lattice cells. For convenience, we imagine these squares to be centered at the integer points, rather than having their vertices at these points. 

\begin{definition}\label{def:poly}
Two lattice cells of $\mathbb{Z}^2$ are \textit{connected} if they share an edge. A \textit{polyomino} or {\it $n$-omino} is a collection of $n$ lattice cells of ${\mathbb{Z}}^{2}$ such that from any cell we can reach any other cell through consecutive connected cells. 
\end{definition}

\begin{definition}
A packing of congruent unit diameter circular disks centered at the points of ${\mathbb{Z}}^{2}$ is called a \textit{digital circle packing} \cite[section 6]{BeKh}. We denote the maximum contact number of such a packing of $n$ circular disks by $c_{{\mathbb{Z}}}(n, 2)$. 
\end{definition}


Recall that $c_{\rm sep}(n, 2) = \csep(\B^2 , n, 2)$. Clearly, every digital circle packing is totally separable and $c_{{\mathbb{Z}}}(n,2)\leq c_{\rm sep}(n, 2)$. Consider a digital packing of $n$ circular disks inscribed in the cells of an $n$-omino. Since each circular disk touches its circumscribing square at the mid-point of each edge and at no other point, it follows that the number of edges shared between the cells of the polyomino equals the contact number of the corresponding digital circle packing. 

Through the rest of this paper $k$, $\ell$ and $\epsilon$ are integers satisfying $\epsilon\in \{0,1\}$ and $0\le k< \ell+\epsilon$. We note that any positive integer $n$ can be uniquely expressed as $n=\ell(\ell+\epsilon)+k$ (as in \cite{AlCe}). We call this the {\it decomposition of $n$}. 

Harary and Harborth \cite{HaHa} studied minimum-perimeter $n$-ominoes and Alonso and Cerf \cite{AlCe} characterized these in $\mathbb{Z}^2$. The latter also constructed a special class of minimum-perimeter polyominoes called basic polyominoes. Let $n=\ell(\ell+\epsilon)+k$. A {\it basic $n$-omino in $\mathbb{Z}^2$} is formed by first completing a quasi-square $Q_{\alpha\times \beta}$ (a rectangle whose dimensions differ by at most $1$ unit) of dimensions $\alpha\times \beta$ with $\{\alpha, \beta\}=\{\ell, \ell+\epsilon\}$ and then attaching a strip $S_{1\times k}$ of dimensions $1\times k$ (resp. $S_{k\times 1}$ of dimensions $k\times 1$) to a vertical side of the quasi-square (resp. a horizontal side of the quasi-square). Here, we assume that the first dimension is along the horizontal direction. Moreover, we denote any of the resulting polyominoes by $Q_{\alpha \times \beta}+ S_{1\times k}$ (resp. $Q_{\alpha \times \beta}+ S_{k\times 1}$). The results from \cite{AlCe,HaHa} indirectly show that $c_{\mathbb{Z}}(n, 2)=\floor{2n-2\sqrt{n}}$, which together with \eqref{eq:bezdek} implies that $\csep(n, 2)=c_{\mathbb{Z}}(n, 2)$. Thus maximal contact digital packings of $n$ circular disks are among maximal contact totally separable packings of $n$ circular disks.

In order to make use of these ideas, we present analogues of polyominoes and digital circle packings in arbitrary normed planes. 

\begin{definition}
Let $\K_{\oo}$ be a smooth $\oo$-symmetric convex domain in $\E^2$ and $\mathbf{P}$ any parallelogram (not necessarily of minimum area) circumscribing $\K_{\oo}$ such that $\K_{\oo}$ touches each side of $\mathbf{P}$ at its midpoint (and not at the corners of $\mathbf{P}$ as $\K_{\oo}$ is smooth). Let $\x$ and $\y$ be the midpoints of any two adjacent sides of $\mathbf{P}$. Then $-\x$ and $-\y$ are also points of $\bd\K_{\oo}\cap\bd\mathbf{P}$. It is easy to see that $\{\x, \y\}$ is an Auerbach basis of the normed plane $(\R^2 , \norm{\cdot}_{\K_{\oo}})$. We call the lattice ${\cal{L}}_\mathbf{P}$ in $(\R^2 , \norm{\cdot}_{\K_{\oo}})$ with fundamental cell $\mathbf{P}$, an {\it Auerbach lattice} of $\K_{\oo}$ as we can think of ${\cal{L}}_\mathbf{P}$ as being generated by the Auerbach basis $\{\x, \y\}$ of $(\R^2 , \norm{\cdot}_{\K_{\oo}})$. 
\end{definition}

On the other hand, any Auerbach basis $\{\x, \y\}$ of a smooth $\oo$-symmetric convex domain $\K_{\oo}$ generates an Auerbach lattice ${\cal{L}}_\mathbf{P}$ of $\K_{\oo}$, with fundamental cell determined by the unique lines supporting $\K_{\oo}$ at $\x$, $\y$, $-\x$ and $-\y$, respectively. In the sequel, we will use ${\cal{L}}_{\mathbf{P}}$ to denote the tiling of $\R^2$ by translates of $\mathbf{P}$ as well as the set of centers of the tiling cells. Indeed, the integer lattice $\mathbb{Z}^2$ is an Auerbach lattice of the circular disk $\B^2$.

\begin{figure}[t]
\centering
\includegraphics[scale=.35]{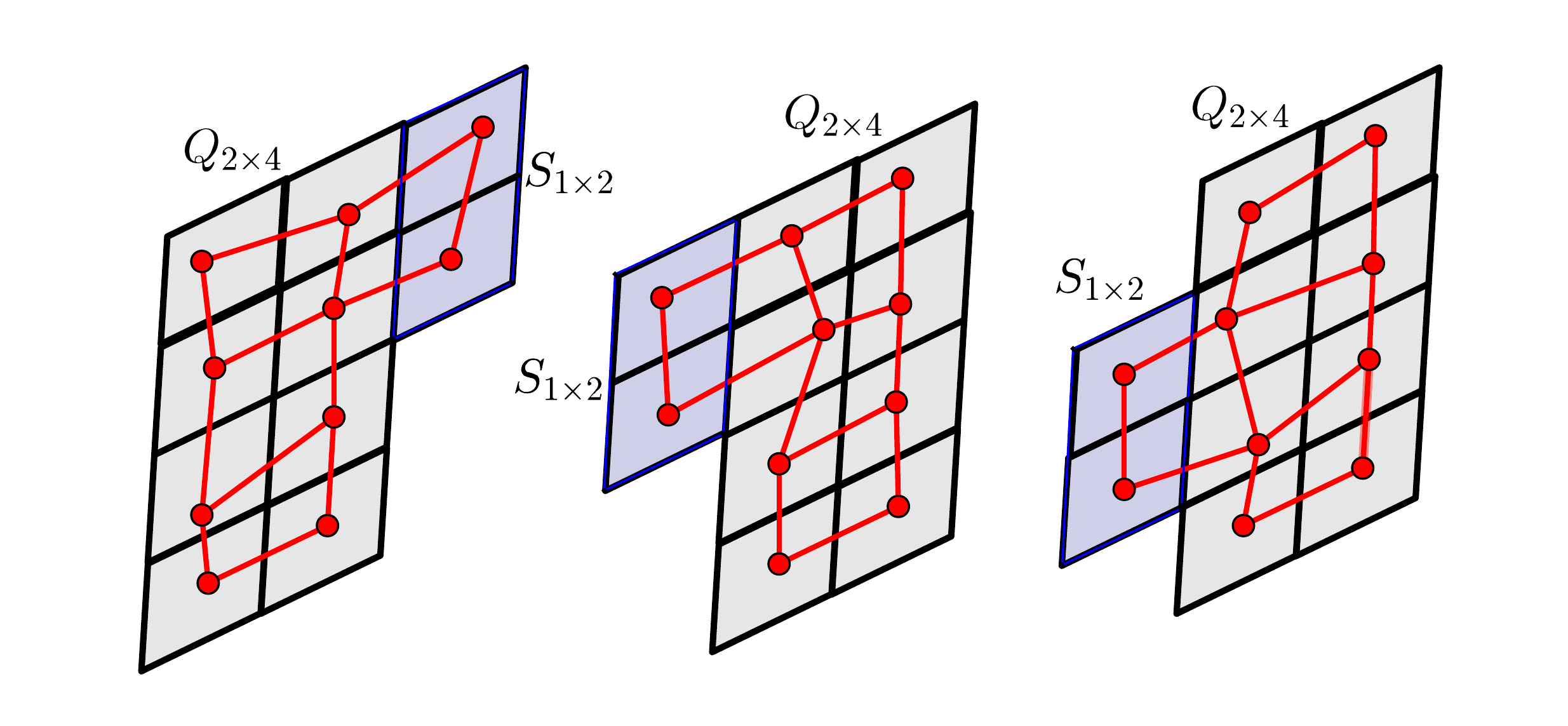}

\vspace{-.5cm}

\caption{Some realizations of the basic polyomino $Q_{2 \times 4}+ S_{1\times 2}$ in some Auerbach lattice and the corresponding graphs.}
\label{fig:polyomino}
\end{figure}

Given an Auerbach lattice ${\cal{L}}_\mathbf{P}$ of a smooth $\oo$-symmetric convex domain $\K_{\oo}\subseteq \E^2$ corresponding to the Auerbach basis $\{\x, \y\}$, we define polyominoes in ${\cal{L}}_\mathbf{P}$ as in Definition \ref{def:poly}. We also define {\it basic $n$-ominoes in ${\cal{L}}_{\mathbf{P}}$} on the same lines as in $\mathbb{Z}^2$ with the first dimension along $\x$, while the second dimension along $\y$. The {\it left} and ${\it right}$ rows of an ${\cal{L}}_{\mathbf{P}}$-polyomino $p$ are defined along $\x$-direction, while the {\it top} and {\it bottom} rows are defined along $\y$-direction in the natural way. The {\it base-lines} of $p$ are the four sides of a minimal area parallelogram containing $p$ and are designated (in a natural way) as the top, bottom, right and left base-line of $p$. The {\it graph of $p$}, denoted by $G(p)$, has a vertex for each cell of $p$, with two vertices adjacent if and only if the corresponding cells share a side. Figure \ref{fig:polyomino} shows some basic polyominoes and their graphs in some Auerbach lattice. We refer to the translates of $\K_{\oo}$ centered at the lattice points of ${\cal{L}}_\mathbf{P}$ (inscribed in the cells of ${\cal{L}}_\mathbf{P}$) as {\it ${\cal{L}}_\mathbf{P}$-translates} of $\K_{\oo}$. Any packing of such translates will be called an {\it ${\cal{L}}_\mathbf{P}$-packing} of $\K_{\oo}$. Since ${\cal{L}}_\mathbf{P}$ is a linear image of $\mathbb{Z}^2$, the results of \cite{AlCe,HaHa}  also hold for ${\cal{L}}_\mathbf{P}$-polyominoes. 

\begin{lemma}\label{digital1}
Let $\K_{\oo}$ be a smooth $\oo$-symmetric convex domain, $n=\ell(\ell+\epsilon)+k$ be the decomposition of a positive integer $n$ and $p$ be an $n$-omino in an Auerbach lattice ${\cal{L}}_\mathbf{P}$ of $\K_{\oo}$. 
\item{(i)} If $\cal{P}$ is a packing of $n$ translates of $\K_{\oo}$ inscribed in the cells of $p$, then $G(p)$ is the contact graph of $\cal{P}$ and therefore, the number of edges in $G(p)$ is equal to the contact number $c({\cal{P}})$ of ${\cal{P}}$. 
\item{(ii)} If in addition $p$ is a minimum-perimeter (or basic) $n$-omino, then $c({\cal{P}})=\floor{2n-2\sqrt{n}}$. 
\end{lemma}

\begin{proof}
Since $\K_{\oo}$ is smooth, no ${\cal{L}}_\mathbf{P}$-translate of $\K_{\oo}$ meets the cell of ${\cal{L}}_\mathbf{P}$ containing it at a corner of the cell. Also any ${\cal{L}}_\mathbf{P}$-translates of $\K_{\oo}$ touches the cell containing it at the midpoints of the four sides of the cell. Therefore, two ${\cal{L}}_\mathbf{P}$-translates of $\K_{\oo}$ touch if and only if the cells of ${\cal{L}}_\mathbf{P}$ containing them share a side. This proves (i). 

Statement (ii) now follows from (i) and \cite{AlCe, HaHa}. 
\end{proof}

We now show in Theorem \ref{circular} that $\csep(\mathbf{D}, n, 2)=\floor{2n-2\sqrt{n}}$, for any smooth $B$-domain $\mathbf{D}$. The existence of a $B$-measure plays a key role in the proof as it provides us a Euclidean-like angle measure to work with. The proof also heavily relies on the ${\cal{L}}_{\mathbf{P}}$-packing ideas discussed above. Smoothness is needed for it allows us to make use of Lemma \ref{digital1} and Theorem \ref{2Dmain} or the following special case of Theorem \ref{2Dmain} that can be proved independently. 

\begin{remark}\label{rem2}
Suppose $\mathbf{D}$ is a smooth $B$-domain, $\mu$ a $B$-measure in $(\R^2 , \norm{\cdot}_{\mathbf{D}})$ and $S$ any separable point set on $\bd \mathbf{D}$. Then for any $\x, \y\in S$, $\mu([\x, \y]_{\mathbf{D}})\ge\pi/2$. Since $\mu(\bd\mathbf{D})=2\pi$, this implies that $\left|S\right|\le4$. By Lemma \ref{lower}, we conclude that $H_{\rm sep}(\mathbf{D})=4$.  
\end{remark}

Moreover, Theorem \ref{circular} is sharp in the sense that it no longer holds if we remove the smoothness assumption. This can be seen through the totally separable packing of $9$ translates of an affine regular convex hexagon given in Figure \ref{fig:lattice}. 

\begin{theorem}\label{circular}
If $\mathbf{D}$ is a smooth $B$-domain in $\E^2$ and $n\geq 2$, then we have 
\begin{equation}\label{eq:contact}
\csep (\mathbf{D},n,2) = \csep (n,2) = \floor{2n-2\sqrt{n}}.
\end{equation}
\end{theorem}

\begin{proof}
First, we establish the lower bound whose proof neither uses smoothness nor the $B$-measure. Consider an Auerbach lattice ${\cal{L}}_\mathbf{P}$ of $\mathbf{D}$ corresponding to an Auerbach basis $\{\x,\y\}$ of $(\R^2, \norm{\cdot}_{\mathbf{D}})$. Then ${\cal{L}}_\mathbf{P} = T(\mathbb{Z}^2 )$, for some linear transformation $T:\R^2 \to \R^2$. Now for any $n\ge 2$, consider a maximal contact digital packing $\cal{C}$ of $n$ circular disks and let $p$ be the corresponding polyomino in $\mathbb{Z}^2$. Then $T(p)$ is an ${\cal{L}}_{\mathbf{P}}$-polyomino with $n$ cells. Let $\cal{P}$ be the packing of ${\cal{L}}_{\mathbf{P}}$-translates of $\mathbf{D}$ inscribed in the cells of $T(p)$, then by Lemma \ref{digital1}, the contact number of $\cal{P}$ is at least as large as the contact number of $\cal{C}$. Thus $\csep(\mathbf{D}, n, 2) \geq \floor{2n-2\sqrt{n}}$.

\begin{figure}[t]
\centering
\includegraphics[scale=.5]{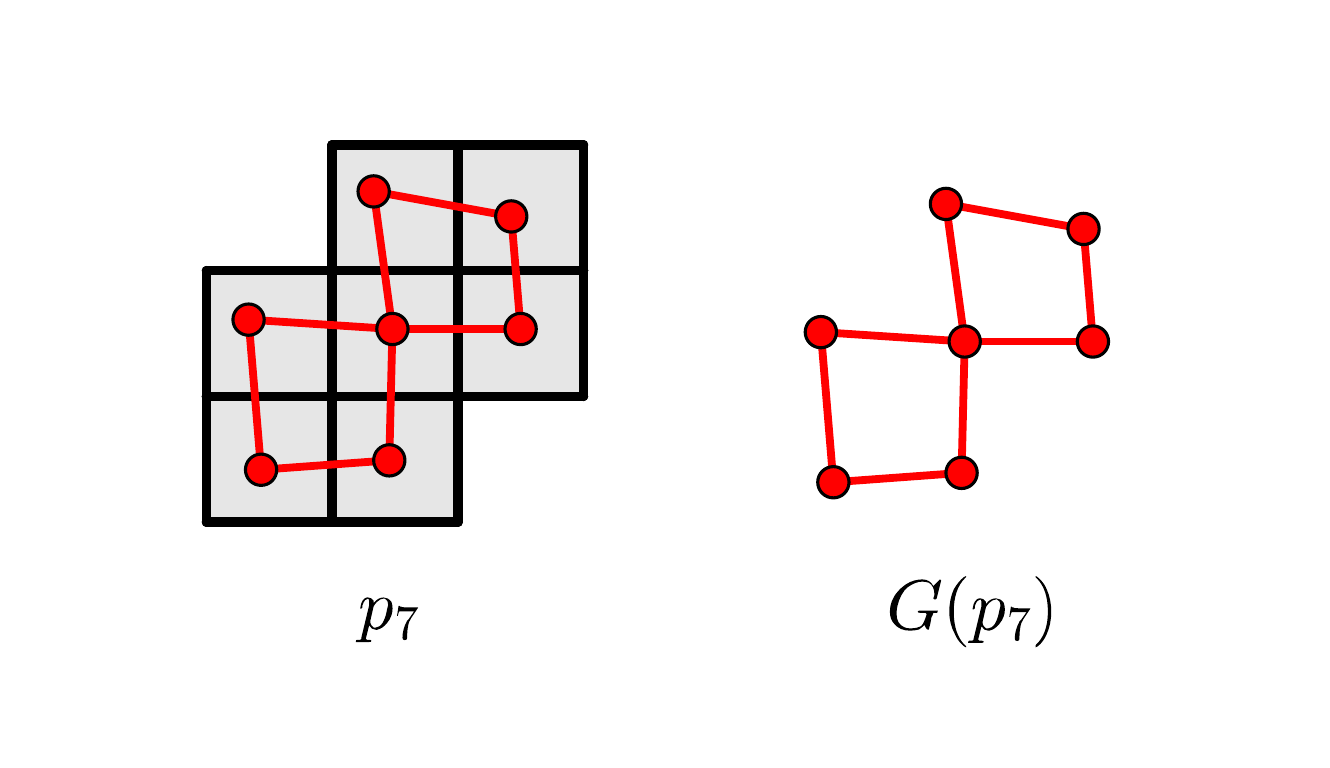}

\vspace{-1cm}

\caption{A $7$-omino $p_7$ whose graph $G(p_7 )$ has minimum degree $2$ but is not $2$-connected.}
\label{fig:pathology}
\end{figure}

Since \eqref{eq:contact} clearly holds for $n\le 3$, for proving the reverse inequality, we proceed by induction on $n$, the number of translates in the packing. Our approach has its origins in the proof outline of \eqref{eq:bezdek} in \cite{BeSzSz}, but requires more sophisticated tools. For the sake of brevity, we write $\csep(n)=\csep(\mathbf{D},n,2)$. Suppose \eqref{eq:contact} is true for totally separable packings of up to $n-1$ translates of $\mathbf{D}$. By Lemma \ref{digital1}, we may assume that for $j\le n-1$, ${\cal{L}}_{\mathbf{P}}$-packings of $j$ translates of $\D$ inscribed in the cells of a basic polyomino in ${\cal{L}}_{\mathbf{P}}$ are among the maximal contact totally separable packings of $j$ translates of $\D$. Let $G$ denote the contact graph of a maximal contact totally separable packing $\cal{P}$ of $n\ge 4$ translates of $\mathbf{D}$. Since $n\ge 2$ and $\csep(n-1)+1=\floor{2(n-1)-2\sqrt{n-1}}+1 \le \floor{2n-2\sqrt{n}}$, we can assume without loss of generality that every vertex of $G$ has degree at least $2$. 

Given a vertex $v$ of $G$, let $\mathbf{D}^v$ denote the corresponding translate in $\cal{P}$. Let $G-v$ denote the graph obtained by deleting $v$ and all the edges incident to $v$ from $G$. Clearly, $G-v$ is the contact graph of the packing ${\cal{P}}\setminus \{\mathbf{D}^v\}$. For a subgraph $H$ of $G$, we define $G-H$ analogously. By performing an affine transformation, if necessary, we may assume that the fundamental cell $\mathbf{P}$ of ${\cal{L}}_{\mathbf{P}}$ is a square. Moreover, we may assume the $\x$-direction to be horizontal and the $\y$-direction to be vertical. This readily allows us to designate the left, right, top and bottom side of a cell (resp. row of a polyomino). 

We now show that except in one exceptional case, $G$ must be $2$-connected. Suppose $n=7$ and $G$ is isomorphic to the graph $G(p_7 )$ of the non-basic polyomino $p_7$ shown in Figure \ref{fig:pathology}. We note that $G(p_7 )$ has minimum degree $2$, but is not $2$-connected. However, this does not cause any issues as  $G$ has $8=\floor{2(7) - 2\sqrt{7}}$ edges and therefore, satisfies the desired upper bound of $\floor{2n - 2\sqrt{n}}$ on the contact number of the underlying packing.  Moreover, none of the contact graphs arising through the rest of the proof is isomorphic to $G(p_7 )$. 

\vspace{2mm}

\noindent Claim 1: If $G$ is not isomorphic to $G(p_7 )$, then $G$ is $2$-connected. 

\vspace{2mm}

Suppose $n=4$. If $u$ is a vertex of $G$ of degree $1$, then ${\cal{P}}\setminus \{\D^u\}$ is a totally separable packing of $3$ translates of $\mathbf{D}$ such that $c({\cal{P}}\setminus \{\D^u\})=3>2=\csep(\mathbf{D}, 3, 2)$, a contradiction. Therefore, $G$ has minimum degree at least $2$. If $v$ is a vertex of $G$ such that $G-v$ is disconnected. Then at least one of the connected components of $G-v$ consists of a single vertex $w$. But then $w$ has degree $1$ in $G$, a contradiction. Therefore $G$ is $2$-connected and we inductively assume $2$-connectedness for contact graphs not isomorphic to $G(p_7 )$ of up to $n-1$ translates of $\D$.


Now suppose that $n=\ell(\ell+\epsilon)+k> 4$ and there exists a vertex $v$ of $G$ such that $G-v$ is disconnected. Let $C$ be one of the connected components of $G-v$ containing the least positive number of neighbours of $v$. Since by Theorem \ref{2Dmain} (or Remark \ref{rem2}), $v$ has degree at most $4$, we encounter two cases. 

\vspace{2mm}

\noindent Case I: The component $C$ contains one neighbour of $v$. 

\vspace{2mm}

\noindent Sub-case I(a): If $G-C-v$ also contains $1$ neighbour of $v$, then we argue as follows. Suppose $C$ contains $n_1$ vertices and so $G-C-v$ consists of $n_2 = n-n_1 - 1$ vertices. Clearly, $n_1$ and $n_2$ are both at least $2$ as otherwise $G$ has a vertex of degree $1$. For $i=1,2$, let $n_i = \ell_i (\ell_i + \epsilon_i ) +k_i$ with $\epsilon_i \in \{0,1\}$, $0\le k_i < \ell_i +\epsilon_i$. 

Let $p_1$ be a realization of $Q_{(\ell_1 +\epsilon_1 )\times \ell_1 }+S_{1\times k_1 }$ with the strip $S_{1\times k_1 }$, if non-empty, forming the right row of $p_1$ such that the top base-lines of $Q_{(\ell_1 +\epsilon_1 )\times \ell_1 }$ and $S_{1\times k_1 }$ coincide. Also let $p_{2}$ be a realization of $Q_{(\ell_2 +\epsilon_2 )\times \ell_2 }+S_{1\times k_2 }$ with the strip $S_{1\times k_2 }$, if non-empty, forming the left row of $p_2$ such that the bottom base-lines of $Q_{(\ell_2 +\epsilon_2 )\times \ell_2 }$ and $S_{1\times k_2 }$ coincide. Denote the bottom-right cell of $p_1$ by $c_1$ and the top-right cell of $p_2$ by $c_2$. Now attach a single cell $c$ to the bottom side of $c_1$ and the right side of $c_2$ and let $p$ be the resulting polyomino as shown in Figure \ref{fig:poly-case-I} (a). Since $\ell_1 +\epsilon_1 \ge 2$, in $p$ there exists at least one cell of $p_1$ other than $c_1$ that shares a side with a cell of $p_2$. Now let ${\cal{P}}_1$ and ${\cal{P}}_2$ be the totally separable ${\cal{L}}_\mathbf{P}$-packings of translates of $\mathbf{D}$ inscribed in the cells of $p_1$ and $p_2$ in $p$, respectively, and ${\cal{P}}'={\cal{P}}_1 \cup {\cal{P}}_2 \cup\{\K_c \}$, where $\K_c$ is the ${\cal{L}}_\mathbf{P}$-translate of $\mathbf{D}$ inscribed in the cell $c$. Then using the inductive assumption about basic polyomino packings, $c({\cal{P}})\le \csep(\mathbf{D}, n_1 , 2) + \csep(\mathbf{D}, n_2 , 2)+ 2 = c({\cal{P}}_1 ) + c({\cal{P}}_2 ) + 2 < c({\cal{P}}_1 ) + c({\cal{P}}_2 )+ 3= c({\cal{P}}')$, a contradiction. 

\begin{figure}[t]
\centering
\includegraphics[scale=.5]{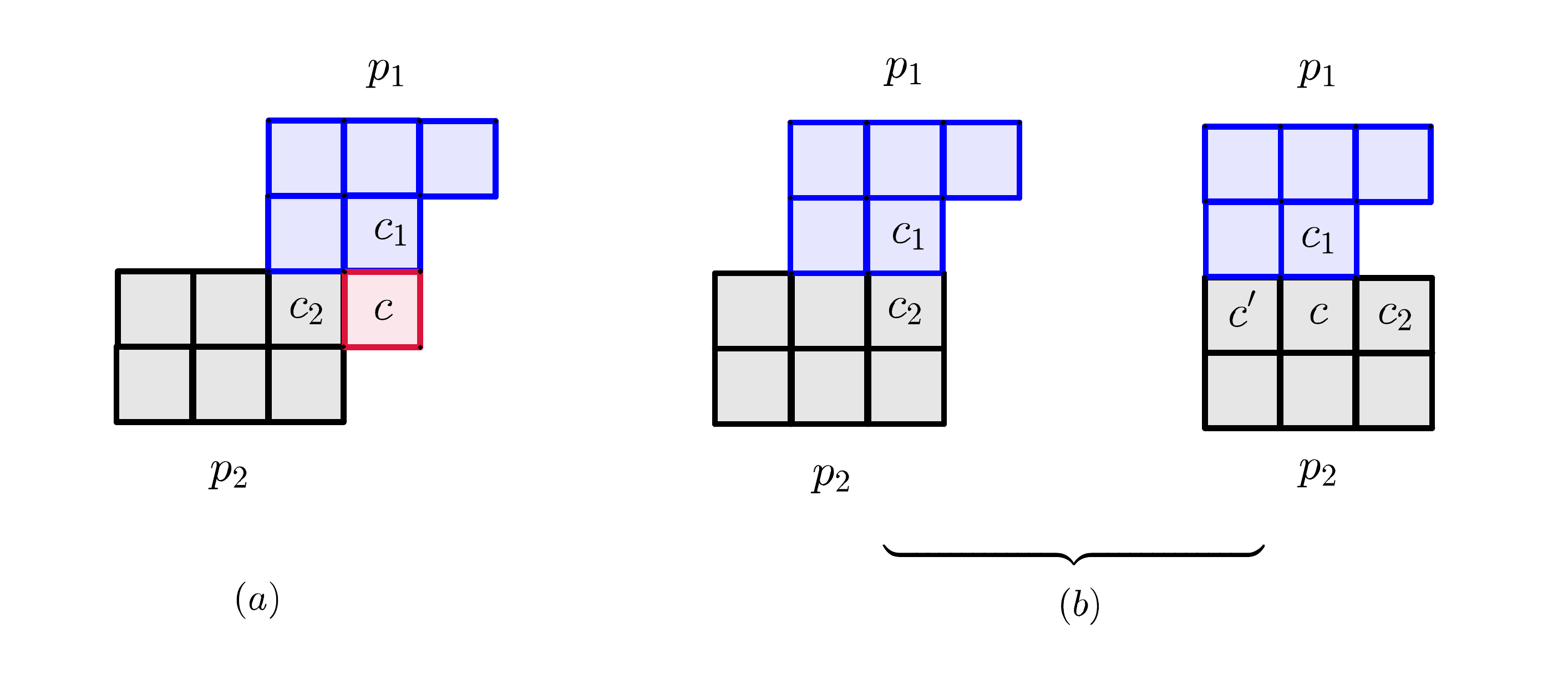}
\vspace{-.8cm}
\caption{Case I in the proof of Theorem \ref{circular} when $G-C-v$ contains (a) $1$ neighbour of $v$, (b) $2$ or $3$ neighbours of $v$.}
\label{fig:poly-case-I}
\end{figure}

\vspace{2mm}

\noindent Sub-case I(b): If $G-C-v$ contains $2$ or $3$ neighbours of $v$, then we modify the proof of I(a) as follows. Suppose $C$ contains $n_1$ vertices and so $G-C$ consists of $n_2 = n-n_1$ vertices. Clearly, $n_1 \ge 2$ and $n_2 \ge 3$ as otherwise $G$ has a vertex of degree $1$. For $i=1,2$, let $n_i = \ell_i (\ell_i + \epsilon_i ) +k_i$ with $\epsilon_i \in \{0,1\}$, $0\le k_i < \ell_i +\epsilon_i$. Note that $n_2$, $\ell_2$, $\epsilon_2$ and $k_2$ do not all have the same values as in Sub-case I(a). Let $p_1$, $p_2$, $c_1$ and $c_2$ be as in the proof of I(a). 

If $G-C-v$ contains $2$ neighbours of $v$, then attach the bottom side of $c_1$ to the top side of $c_2$ to form a polyomino $p$. Since $\ell_1 +\epsilon_1 \ge 2$ and $\ell_2 +\epsilon_2 \ge 2$, in $p$ there exists at least one cell of $p_1$ other than $c_1$ that shares a side with a cell of $p_2$. 

On the other hand, if $G-C-v$ contains $3$ neighbours of $v$ in $G$, say $x$, $y$ and $z$, then by the total separability of $\cal{P}$ and smoothness of $\mathbf{D}$, no two of $\mathbf{D}^x$, $\mathbf{D}^y$ and $\mathbf{D}^z$ touch. Since minimum degree of $G$ is at least two, each of the vertices $x$, $y$ and $z$ have a neighbour in $G-C-v$. Moreover, again by the total separability of $\cal{P}$ and smoothness of $\mathbf{D}$, $x$, $y$ and $z$ cannot have the same vertex of $G-C-v$ as a common neighbour. Therefore, $n_2 \ge 6$ and $\ell_2 + \epsilon_2 \ge 3$. Thus the top row of $p_2$ consists of at least $3$ cells $c_2$, $c$ and $c'$ ordered from right to left. We attach the bottom side of $c_1$ to the top side of $c$ to form a polyomino $p$. Since $\ell_1 +\epsilon_1 \ge 2$, there exists a cell of $p_1$ other than $c_1$ that shares its bottom side with the top side of $c'$.

\begin{figure}[t]
\centering
\includegraphics[scale=.5]{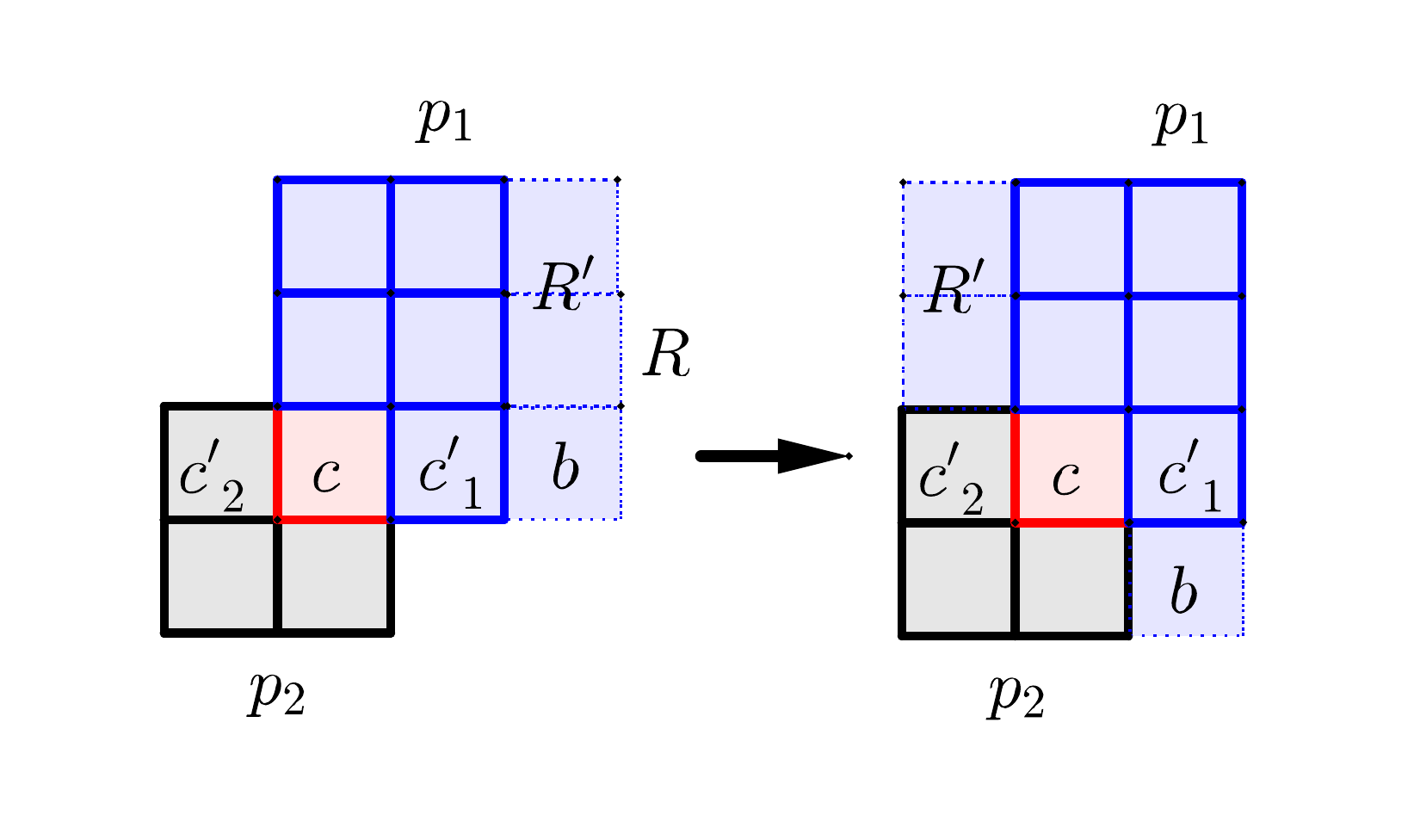}
\vspace{-1cm}
\caption{Case II in the proof of Theorem \ref{circular} when $p_1$ contains more cells than $p_2$.}
\label{fig:poly-case-IIa}
\end{figure}

In either scenario (illustrated in Figure \ref{fig:poly-case-I} (b)), let ${\cal{P}}_1$ and ${\cal{P}}_2$ be the totally separable ${\cal{L}}_\mathbf{P}$-packings of translates of $\mathbf{D}$ inscribed in the cells of $p_1$ and $p_2$ in $p$ and let ${\cal{P}}'={\cal{P}}_1 \cup {\cal{P}}_2 $. Then again using the inductive assumption about basic polyomino packings, $c({\cal{P}})\le \csep(\mathbf{D}, n_1 , 2) + \csep(\mathbf{D}, n_2 , 2)+ 1 = c({\cal{P}}_1 ) + c({\cal{P}}_2 ) + 1 < c({\cal{P}}_1 ) + c({\cal{P}}_2 )+ 2 = c({\cal{P}}')$, a contradiction. 

\vspace{2mm}

\noindent Case II: The component $C$ contains two neighbours of $v$. 

\vspace{2mm}

In this case, assume that $C$ consists of $n_1 -1$ vertices, so $G-C$ contains $n_2 = n-n_1 + 1$ vertices. Let $x$ and $y$ be the two neighbours of $v$ in $C$. Then by the total separability of the packing and smoothness of $\mathbf{D}$, $\D^x$ and $\D^y$ do not touch. However, since every vertex in $G$ has degree at least $2$, there exists at least one vertex $z\ne x, y$ in $C$. Thus $n_1 \ge 4$. A similar argument shows that $n_2 \ge 4$ and therefore, $n\ge 7$. Using the degree condition on $v$, the disconnectedness of $G-v$ and the total separability of the underlying packing, we observe that $G(p_7 )$ is the only $7$-vertex contact graph on which Case II applies. Since, $G$ is not isomorphic to $G(p_7 )$, we have $n\ge 8$. For $i=1,2$, let $n_i = \ell_i (\ell_i + \epsilon_i ) +k_i$ with $\epsilon_i \in \{0,1\}$, $0\le k_i < \ell_i +\epsilon_i$ and note that $n_1$, $\ell_1$, $\epsilon_1$ and $k_1$ do not all have the same values as in Sub-case I(b). Let $p_1$, $p_2$ and $c_2$ be as in the proof of I(a) and I(b), but $c_1$ be the bottom-left cell of $p_1$. Since $\ell_1 +\epsilon_1 \ge 2$, there exists a cell $c'_1$ adjacent to $c_1$ in the bottom row of $p_1$. Also since $\ell_2 +\epsilon_2 \ge 2$, there exists a cell $c'_2$ adjacent to $c_2$ in the top row of $p_2$. 

As $n\ge 8$, at least one of $p_1$ and $p_2$ consists of at least $5$ cells and at least $3$ vertical rows. Form an $n$-omino by overlapping the cells $c_1$ and $c_2$ into a single cell $c$. If $p_1$ contains more cells than $p_2$, translate the bottom cell $b$ of the right row $R$ of $p_1$ down and left and attach it to the bottom side of $c'_1$. Let $R'$ denote the rest of $R$. We move $R'$ so that the bottom side of $R'$ is attached to the top side of $c'_2$. Note that each cell of $R'$ shares its right side with a cell in the left row of $p_1$ as shown in Figure \ref{fig:poly-case-IIa}. If $p_2$ contains more cells than $p_1$, translate the top cell $t$ of the left row $L$ of $p_2$ up and right and attach it to the top side of $c'_2$. Let $L'$ denote the rest of $L$. We move $L'$ so that the top side of $L'$ is attached to the bottom side of $c'_1$. Note that each cell of $L'$ shares its left side with a cell in the right row of $p_2$ as shown in Figure \ref{fig:poly-case-IIb}. 

\begin{figure}[t]
\centering
\includegraphics[scale=.5]{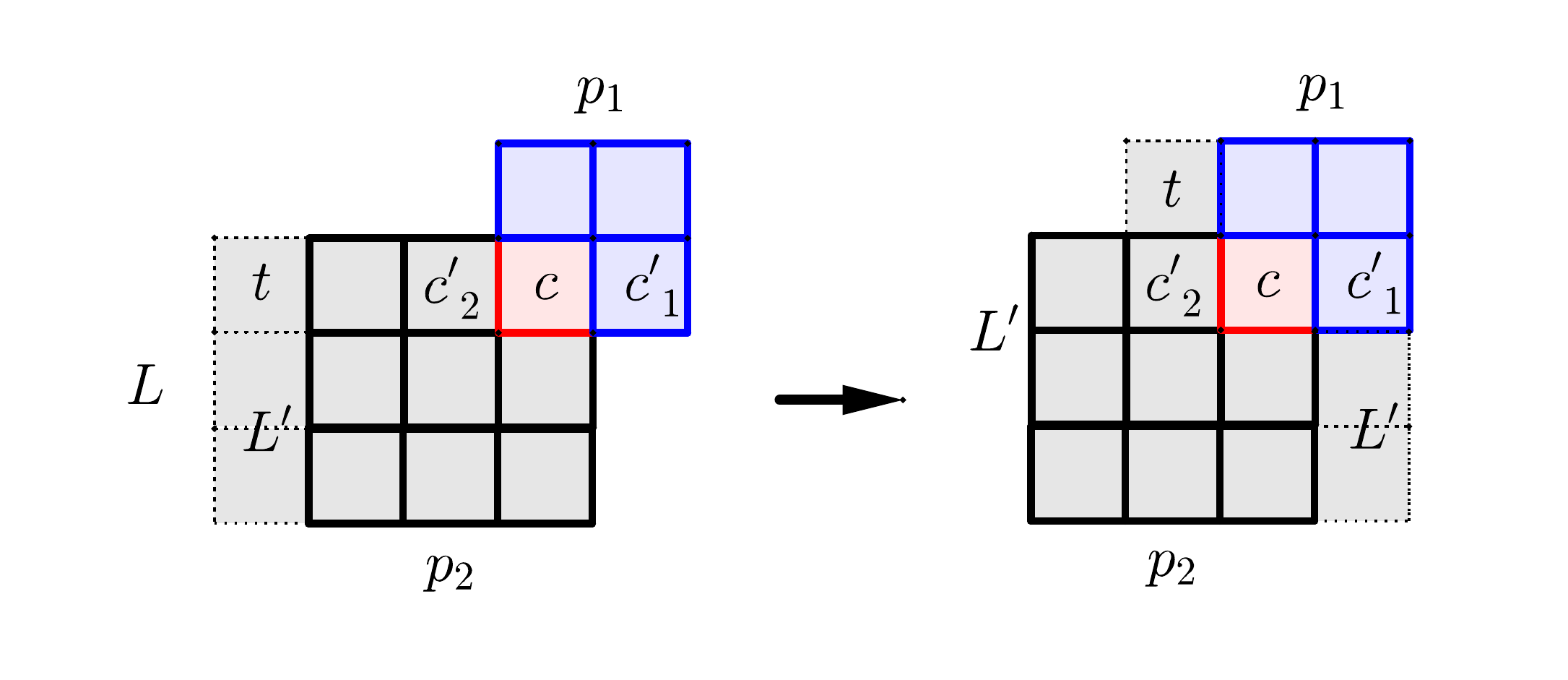}
\vspace{-1cm}
\caption{Case II in the proof of Theorem \ref{circular} when $p_2$ contains more cells than $p_1$.}
\label{fig:poly-case-IIb}
\end{figure}

Let $p$ be the resulting $n$-omino and ${\cal{P}}'$ be the totally separable ${\cal{L}}_{\mathbf{P}}$-packing of translates of $\mathbf{D}$ inscribed in the cells of $p$. Clearly, $c({\cal{P}})+1\le  c({\cal{P}}')$, a contradiction. This proves Claim 1.

Thus $G$ is a $2$-connected planar graph with $\csep(n)$ edges having minimum vertex degree at least $2$ and so, every face of $G$ -- including the external one -- is bounded by a cycle. Thus $G$ is bounded by a simple closed polygon $P$. Let $v$ denote the number of vertices of $P$. By Theorem \ref{2Dmain} (or Remark \ref{rem2}), the degree of each vertex in $G$ is $2$, $3$ or $4$. For $j\in \{2,3,4\}$, let $v_{j}$ be the number of vertices of $P$ of degree $j$. By definition of $B$-domains, there exists a $B$-measure $m$ in $(\R^2 , \norm{\cdot}_{\mathbf{D}})$ so that using total separability of our packing, the internal angle of $P$ at a vertex of degree $j$ is at least $\frac{(j-1)\pi}{2}$. Since the internal angle sum formula holds for angular measures, the sum of these angles will be $(v-2)\pi$. Clearly $v=v_{2}+v_{3}+v_{4}$, and thus we get the inequality
\begin{equation} \label{5eq1}
v_{2}+2v_{3}+3v_{4}\leq 2v-4.
\end{equation}

Now let $g_{j}$ be the number of internal faces of $G$ that have $j$ sides. By total separability and smoothness, $j\geq 4$. It follows from Euler's polyhedral formula that
\begin{equation} \label{5eq2}
n-\csep(n)+g_{4}+g_{5}+\ldots =1.
\end{equation}

In the process of adding up the number of sides of the internal faces of $G$, every edge of $P$ is counted once and all the other edges are counted twice. Therefore, 
\begin{equation}\label{5eq3}
4(g_{4}+g_{5}+\ldots)\leq 4g_{4}+5g_{5}+\ldots =v+2(\csep(n)-v).
\end{equation} 

This together with \eqref{5eq2} implies that $4(1-n+\csep(n))\leq v+2(\csep(n)-v)$, and thus we obtain
\begin{equation}\label{5eq4}
2\csep(n)-3n+4\leq n-v.
\end{equation}

From $G$, delete the vertices of $P$ along with the edges that are incident to them. From the definition of $\csep(n-v)$, we get $\csep(n)-v-(v_{3}+2v_{4})\leq \csep(n-v)$, which together with (\ref{5eq1}) implies that $
\csep(n)\leq \csep(n-v)+2v-4$. By induction hypothesis, $\csep(n-v)\leq 2(n-v)-2\sqrt{n-v}$, and so 
\begin{equation}\label{5eq5} 
\csep(n)\leq (2n-4)-2\sqrt{n-v}. 
\end{equation}
Using (\ref{5eq4}) it follows that $\csep(n)\leq (2n-4)-2\sqrt{2\csep(n)-3n+4}$, and so 
\[
\csep(n)^{2}-4n\csep(n)+(4n^{2}-4n)\ge 0.
\]

Finally, since the solutions of the quadratic equation $x^{2}-4nx+(4n^{2}-4n)=0$ are $x=2n\pm 2\sqrt{n}$ and $\csep(n)<2n$, therefore it follows that $\csep(n)\leq 2n-2\sqrt{n}$.
\end{proof}

The arguments given in the proof of lower bound in Theorem \ref{circular} establish the following lower bound on the maximum separable contact number of $n$ translates of any convex domain. Note that we can drop the assumption of $\oo$-symmetry due to Lemma \ref{minkowski}. 

\begin{remark}\label{smooth-contact}
For any convex domain $\K$, we have $\csep (\K,n,2) \ge \floor{2n-2\sqrt{n}}$. 
\end{remark}

\subsection{Smooth $A$-domains and their maximum separable contact numbers}\label{sec:adomain}
Before we give the definition of an $A$-domain, note that although Birkhoff orthogonality is a non-symmetric relation in general, it turns out to be symmetric in Euclidean spaces. That is, for any $\x, \y\in \E^d$, $\x\dashv_{\B^d}\y$ holds if and only if $\y\dashv_{\B^d}\x$ holds. 

\begin{definition}\label{a-domain}
Let $\mathbf{A}\subseteq \E^2$ be an $\oo$-symmetric convex domain, $\B$ a circular disk centered at $\oo$ and $c, -c, c', -c'$ non-overlapping arcs on $\bd \B \cap \bd \mathbf{A}$ such that for any $\x\in c$ there exists $\x'\in c'$ with $\x\dashv_{\B}\x'$ and vice versa. Then we call $\mathbf{A}$ an Auerbach domain, or simply an $A$-domain.  
\end{definition}

\begin{figure}[t]
\centering
\includegraphics[scale=.7]{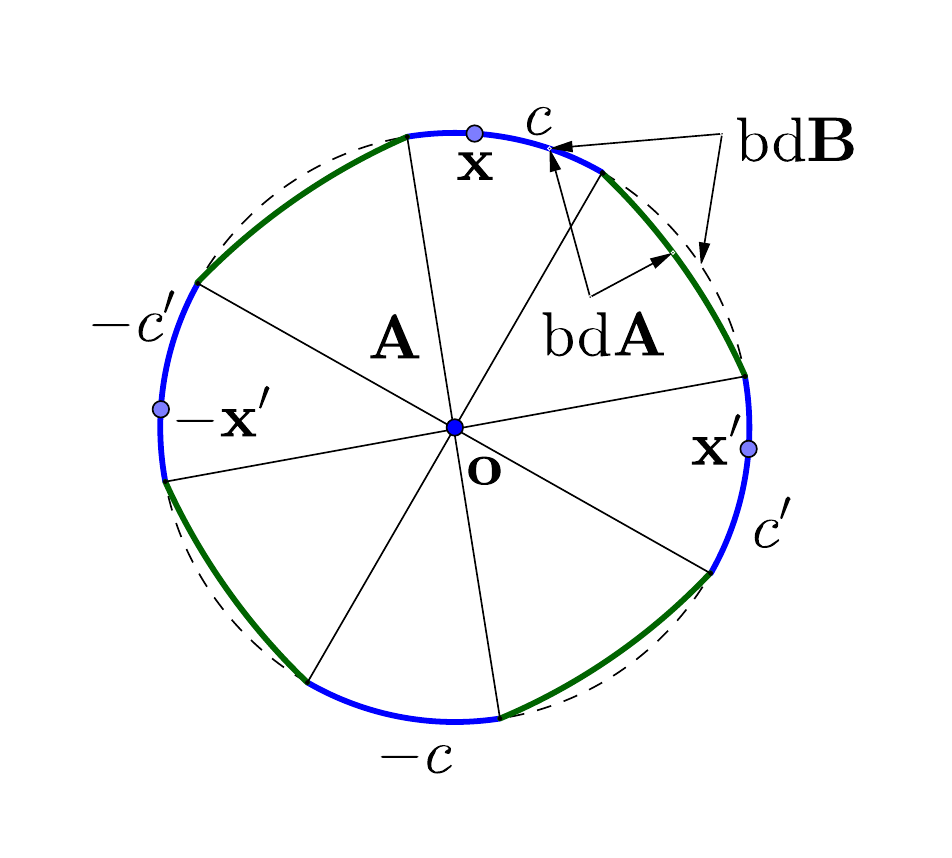}\ \ \ \ \includegraphics[scale=.7]{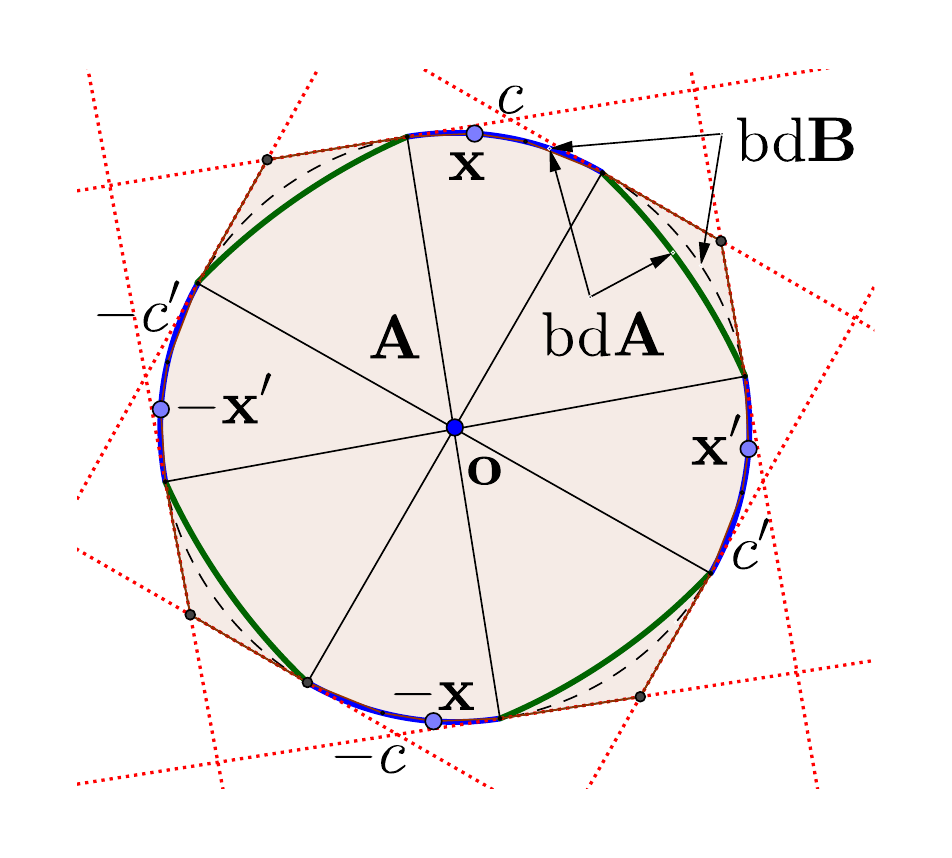}

\vspace{-.7cm}

\caption{Definition \ref{a-domain} illustrated. An $A$-domain $\mathbf{A}$ with circular pieces $c$, $c'$, $-c$ and $-c'$ lying on the boundary of the circular disk $\B$. Note that $\mathbf{A}$ is not necessarily contained in $\B$. However, due to convexity, $\mathbf{A}$ must lie in the shaded region determined by the tangent lines at the end points of the circular pieces.}
\label{fig:a-domain}
\end{figure}



Figure \ref{fig:a-domain} illustrates Definition \ref{a-domain}. Clearly, for any $\x\in c$ there exist antipodes $\x', -\x'\in \bd \B$ with $\x\dashv_{\B}\x'$ and $\x\dashv_{\B}-\x'$. From Definition \ref{a-domain}, we must have $\x'\in c'$ and $-\x'\in -c'$. Moreover, an analogous statement holds for any $\x'\in c'$. Therefore, an $A$-domain $\mathbf{A}$ can be thought of as an $\oo$-symmetric convex domain in $\E^2$ such that $\bd \mathbf{A}$ contains two pairs of antipodal circular arcs all lying on the same circle and with each pair being Birkhoff orthogonal to the other in $\E^2$. Note that this definition does not exclude the case when more than one set of such arcs occurs on $\bd \mathbf{A}$, in which case we choose the set of four arcs arbitrarily, or even when $\mathbf{A}$ is a circular disk. Given an $A$-domain $\mathbf{A}$, we call the four circular arcs $c, -c, c', -c'$ chosen on its boundary, the {\it circular pieces} of $\mathbf{A}$ and write 
$${\rm cir}(\mathbf{A}) = c\cup (-c) \cup c'\cup (-c').$$ 
Clearly, if $\x, \y \in {\rm cir}(\mathbf{A})$, then $\x\dashv_{\mathbf{A}}\y$ holds if and only if $\x\dashv_{\B}\y$ holds. We observe that any $\x\in {\rm cir}(\mathbf{A})$ belongs to an Auerbach basis of $\mathbf{A}$. Furthermore, any Auerbach basis of $\mathbf{A}$ is either contained in ${\rm cir}(\mathbf{A})$ or $\bd \mathbf{A}\setminus {\rm cir}(\mathbf{A})$. 

The reason behind defining $A$-domains is two-fold. First, in the next section we prove that any smooth $\oo$-symmetric strictly convex domain can be approximated arbitrarily closely by an $A$-domain. This leads to an exact computation of the maximum separable contact number for any packing of $n$ translates of a smooth strictly convex domain. Second, one can explicitly construct a $B$-measure on the boundary of each $A$-domain that very closely mirrors the properties of the Euclidean angle measure. We give the construction below.  

Let $\mathbf{A}$ be an $A$-domain with circular pieces $c$, $-c$, $c'$ and $-c'$ lying on the boundary of a circular disk $\B$ and $e$ denote the Euclidean angle measure. Then $e(c) = e(c') = e(-c) = e(-c')$ holds and we define an angle measure $m$ on $\bd \mathbf{A}$ as follows. For any arc $a\subseteq \bd \mathbf{A}$, define  
\begin{equation}\label{eq:angle}
m(a) = 2\pi\frac{e(a\cap {\rm cir}(\mathbf{A}))}{e({\rm cir}(\mathbf{A}))}.
\end{equation}

Note that $m$ assigns a measure of $\pi/2$ to each of the designated circular pieces on $\bd \mathbf{A}$ and a measure of $0$ to the rest of $\bd \mathbf{A}$ (including any circular arcs not included among the circular pieces), that is, $m(c) = m(c') = m(-c) = m(-c')=\pi/2$ and $m(\bd \mathbf{A} \setminus {\rm cir}(\mathbf{A})) = 0$. It is easy to check that $m$ satisfies properties (i-iii) of Definition \ref{def:angle}, as well as the following property. 

\begin{lemma}\label{b-measure}
Let $\mathbf{A}$ be an $A$-domain and $m$ the angle measure defined on $\bd \mathbf{A}$ by \eqref{eq:angle}. Also, let $\x , \y \in \bd \mathbf{A}$ be such that $\x\dashv_{\mathbf{A}}\y$. Then 
\begin{equation}\label{eq:property}
m([\x , \y]_{\mathbf{A}}) = \frac{\pi}{2}.
\end{equation}
In other words, $m$ is a $B$-measure in $(\R^{2}, \norm{\cdot}_{\mathbf{A}})$ and every $A$-domain is a $B$-domain. 
\end{lemma}

\begin{proof}
Let $c$, $-c$, $c'$ and $-c'$ be the circular pieces of $\mathbf{A}$ lying on the boundary of a circular disk $\B$. Let $\x, \y\in \bd\mathbf{A}$ be such that $\x\dashv_{\mathbf{A}}\y$. Two cases arise. If $\x\in {\rm cir}(\mathbf{A})$, then $\x\dashv_{\B}\y$ holds. Without loss of generality assume that $\x\in c$. Then either $\y\in c'$ or $\y\in -c'$, and $m([\x , \y]_{\mathbf{A}}) = m([\x , \y]_{\mathbf{A}}\cap {\rm cir}(\mathbf{A})) = m(c) = \pi/2$. If $\x\in \bd \mathbf{A}\setminus {\rm cir}(\mathbf{A})$, then necessarily $\y\in \bd \mathbf{A}\setminus {\rm cir}(\mathbf{A})$ and $[\x , \y]_{\mathbf{A}}$ contains exactly one of the circular pieces of $\mathbf{A}$, say $c$. Thus once again, $m([\x , \y]_{\mathbf{A}}) = m(c) = \pi/2$.
\end{proof}


As $\csep(\cdot, n, 2)$ is invariant under affine transformations, Theorem \ref{circular} and Lemma \ref{b-measure} give: 

\begin{corollary}\label{a-dom}
If $\mathbf{A}$ is (an affine image of) a smooth $A$-domain in $\E^2$ and $n\geq 2$, we have $\csep (\mathbf{A},n,2) = \csep (n,2) = \floor{2n-2\sqrt{n}}$.
\end{corollary}

\subsection{Maximum separable contact numbers of smooth strictly convex domains}\label{sec:strictly}
We begin with the definition of Hausdorff distance between two convex bodies. 

\begin{definition}
Given two (not necessarily $\oo$-symmetric) convex bodies $\K$ and $\LL$ in $\E^d$, the Hausdorff distance between them is defined as  
\[h(\K, \LL)= \min \left\{\epsilon : \K\subseteq \LL+\epsilon\B^d , \LL\subseteq \K+\epsilon\B^d \right\}. 
\]
\end{definition}

It is well-known that $h(\cdot , \cdot)$ is a metric on the set of all $d$-dimensional convex bodies \cite[page 61]{Sc}. The first main result of this section shows that we can approximate any smooth $\oo$-symmetric strictly convex domain in $\E^2$ arbitrarily closely by an affine image of an $A$-domain with respect to the Hausdorff distance. In fact, the following theorem proves a stronger statement. 

\begin{figure}[t]
\centering
\includegraphics[scale=.6]{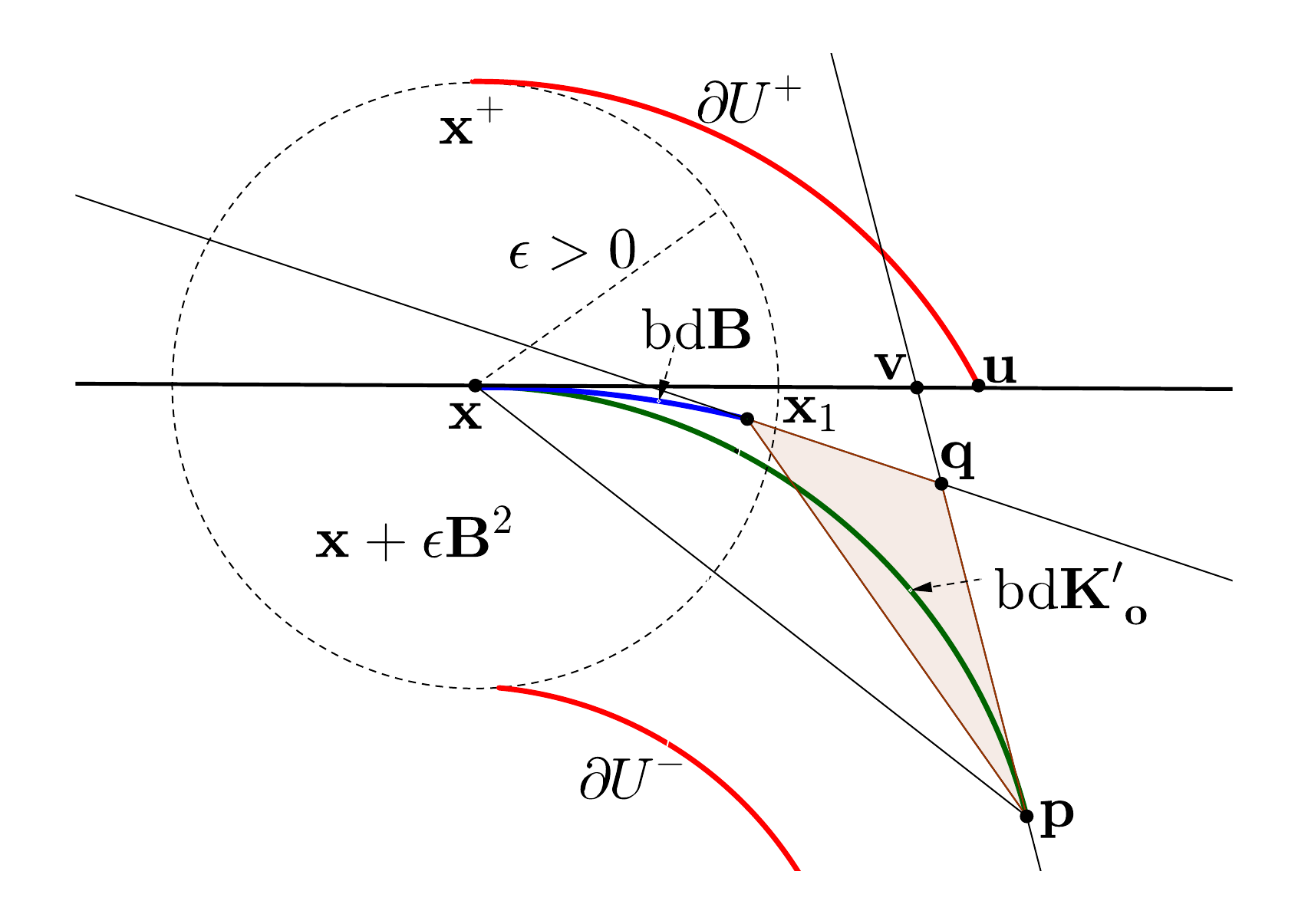}

\vspace{-.5cm}

\caption{Replacing a part of the boundary of the smooth $\oo$-symmetric strictly convex domain $\K'_{\oo}$ in the proof of Theorem \ref{dense} by a circular arc. The circular arc on $\bd\B$ is colored blue, while the green arc represents a part of the boundary of $\K_{\oo}'$. The red arcs represent the outer and inner boundary of the annulus $\bd \K_{\oo}' +\epsilon\B^2$. The construction is independent of whether $\x_1 \in \K_{\oo}'$ or not.}
\label{fig:circle}
\end{figure}

\begin{theorem}\label{dense}
Affine images of smooth strictly convex $A$-domains are dense (in the Hausdorff sense) in the space of smooth $\oo$-symmetric strictly convex domains. Moreover, given any smooth $\oo$-symmetric strictly convex domain $\K_{\oo}$, we can construct an affine image $\mathbf{A}'$ of a smooth strictly convex $A$-domain $\mathbf{A}$ such that the length of $\bd\mathbf{A}' \cap\bd \K_{\oo}$ can be made arbitrarily close to the length of $\bd\K_{\oo}$. 
\end{theorem}

\begin{proof}
Let $\K_{\oo}$ be a smooth $\oo$-symmetric strictly convex domain and $\epsilon>0$ be sufficiently small. We describe the construction of a smooth strictly convex $A$-domain $\mathbf{A}$ with the property that $h(\mathbf{A}', \K_{\oo})\le \epsilon$, for some image $\mathbf{A}'=T(\mathbf{A})$ of $\mathbf{A}$ under an invertible linear transformation $T:\E^2 \to \E^2$. Let $\K_{\oo}' := T^{-1}(\K_{\oo})$. We note that $\K_{\oo}'$ is a smooth $\oo$-symmetric strictly convex domain in $\E^2$. It is sufficient to show that $\bd \mathbf{A}$ lies in the annulus $\bd \K_{\oo}'+\epsilon \B^2$ (see Figure \ref{fig:circle}) and the length of $\bd\mathbf{A}\cap\bd \K_{\oo}'$ can be made arbitrarily close to the length of $\bd\K_{\oo}'$ during the construction. 

We choose $T$ so that the minimal area parallelogram $\mathbf{P}$ containing $\K_{\oo}'$ is is a square. Let $\x, \y, -\x, -\y\in \bd \K_{\oo}'$ be the midpoints of the sides of $\mathbf{P}$. Then $\{\x, \y\}$ is an Auerbach basis of $(\R^2 , \norm{\cdot}_{\K_{\oo}'})$ \cite{Pl} and, by strict convexity, $\bd \K_{\oo}'$ intersects $\bd \mathbf{P}$ only at points $\x$, $-\x$, $\y$ and $-\y$. Let $\B$ be the circular disk centered at $\oo$ that touches $\bd \mathbf{P}$ at the points $\x$, $\y$, $-\x$, $-\y$. Without loss of generality, we may assume that the side of $\mathbf{P}$ passing through $\x$ is horizontal and $\y$ lies on the clockwise arc on $\bd \K_{\oo}'$ from $\x$ to $-\x$. 

Let $U_\epsilon = \bd \K_{\oo}'+\epsilon \B^2$ be the $\epsilon$-annular neighbourhood of $\bd\K_{\oo}'$ with outer boundary curve $\partial U^+$ and inner boundary curve $\partial U^-$. Let $\x^+$ be the unique point of intersection of $\partial U^+$ and $\x+ \epsilon \B^2$. Moving clockwise along $\partial U^+$ starting from $\x^+$, let $\mathbf{u}$ be the first point where $\partial U^+$ intersects $\bd \mathbf{P}$. Starting from $\x$ and moving along $\bd \K_{\oo}'$ clockwise, choose a point $\p\in \bd \K_{\oo}'$ so that the tangent line supporting $\K_{\oo}'$ at $\p$ intersects $(\x, \mathbf{u})$. This unique point of intersection is represented by $\mathbf{v}$ in Figure \ref{fig:circle}. Note that by strict convexity of $\K_{\oo}'$, such a point $\p$ necessarily exists (for example, choose the point on $\bd \K_{\oo}'$ directly below $\mathbf{u}$) and any such $\p$ can be replaced by any point on the open arc $(\x, \p)_{\K_{\oo}'}$. Now choose a point $\x_1 \in \bd \B$ close to $\x$ in clockwise direction so that the line tangent to $\B$ at $\x_1$ intersects $(\p,\mathbf{v})$. In Figure \ref{fig:circle}, $\mathbf{q}$ denotes this point of intersection. Again note that such a point $\x_1$ necessarily exists as the line supporting $\B$ at $\x$ is horizontal. Moreover, $\x_1$ can be replaced by any point on the open arc $(\x, \x_1 )_{\B}$. Therefore, we may assume that $\x_1 \ne \mathbf{q}$ and so $\p$, $\q$ and $\x_1$ form a triangle $\Delta$. Thus there exists a (actually, infinitely many) smooth strictly convex curve $S\subseteq \Delta$ with endpoints $\x_1$ and $\p$ such that the convex domain 
\[\mathbf{A}_1 = \conv((\bd\K_{\oo}'\setminus ([\x, \p]_{\K_{\oo}'}\cup [-\x, -\p]_{\K_{\oo}'}))\cup S \cup[\x, \x_1 ]_{\B}\cup (-S) \cup [-\x, -\x_1 ]_{\B})
\]
obtained by replacing the antipodal boundary arcs $[\x, \p]_{\K_{\oo}'}$ and $[-\x, -\p]_{\K_{\oo}'}$ of $\K_{\oo}'$ with the antipodal circular arcs $[\x, \x_1 ]_{\B}$ and $[-\x, -\x_1 ]_{\B}$ and the smooth and strictly convex connecting curves $S$ and $-S$ is a smooth $\oo$-symmetric strictly convex domain with $\bd \mathbf{A}_1 \subseteq \bd \K_{\oo}' + \epsilon \B^2$. Repeat the procedure for $\mathbf{A}_1$, but this time move counterclockwise along $\bd \mathbf{A}_1$ starting from $\x$. The result is another smooth $\oo$-symmetric strictly convex domain $\mathbf{A}_2$ with $\bd \mathbf{A}_2 \subseteq \bd \K_{\oo}' + \epsilon \B^2$. 

\begin{figure}[t]
\centering
\hspace{-3em}\includegraphics[scale=1]{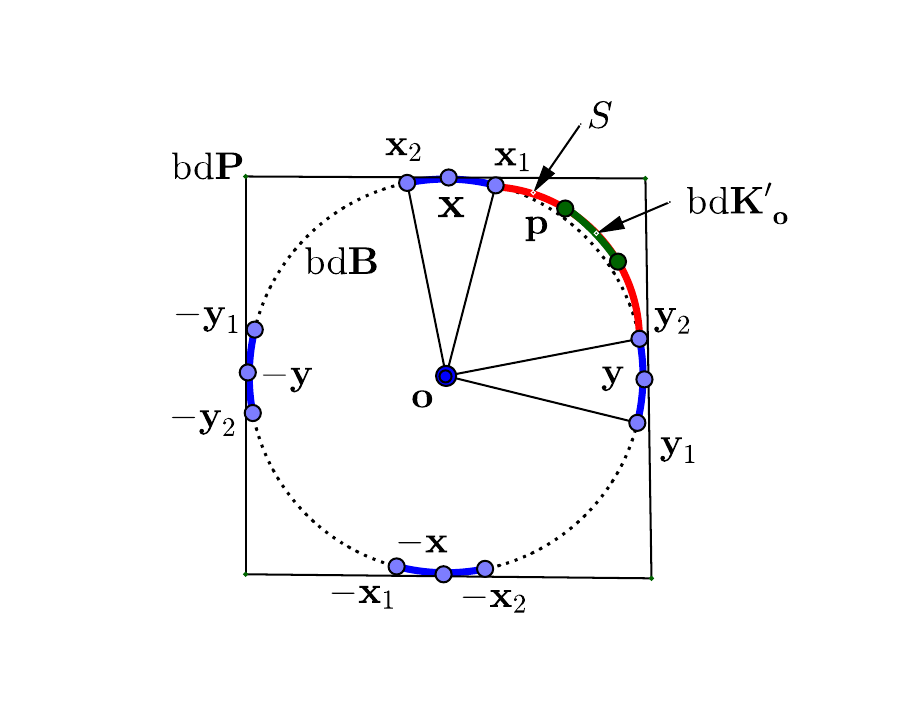}\hspace{-3em}\includegraphics[scale=1]{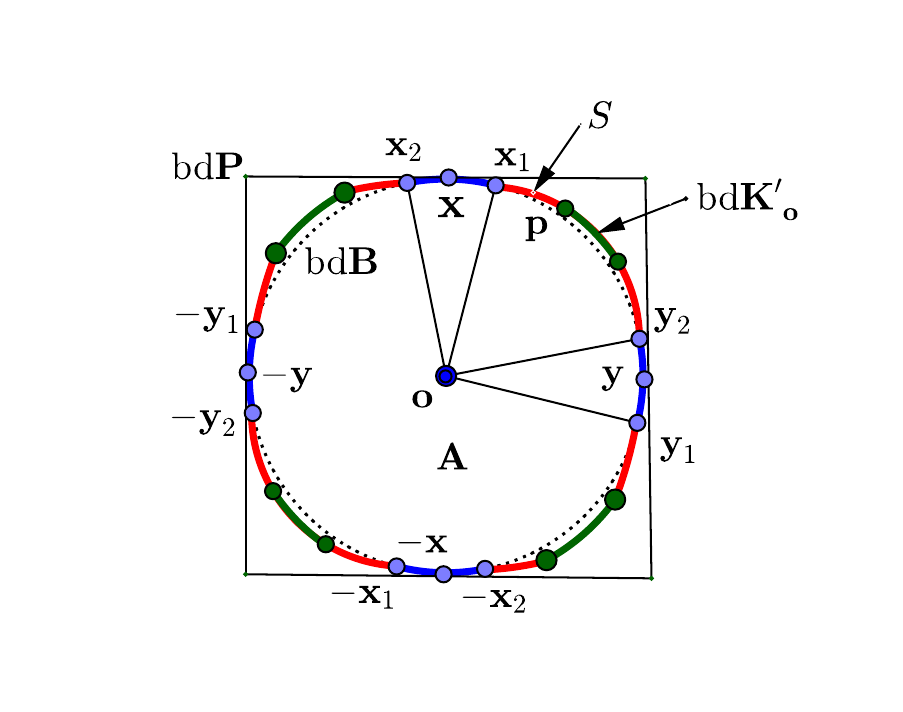}

\vspace{-1cm}

\caption{Construction of a smooth strictly convex $A$-domain approximating the smooth $\oo$-symmetric strictly convex domain $\K_{\oo}'$ in the proof of Theorem \ref{dense}. The circular arcs are colored blue, while the green arcs represent parts of the boundary of $\K_{\oo}$. The red arcs represent smooth strictly convex connections.}
\label{fig:circle2}
\end{figure}

Let $[\x_1 , \x_2 ]_{\B}\subseteq \bd \mathbf{A}_2$ be the counterclockwise circular arc containing $\x$ (but not necessarily centered at $\x$) obtained in this way. We say that $[\x_1 , \x_2 ]_{\B}$ is a {\it replacement arc} for $\K_{\oo}'$ at $\x$ (and therefore, $[-\x_1 , -\x_2 ]_{\B}$ is a replacement arc for $\K_{\oo}'$ at $-\x$). Let $\y_1 , \y_2 \in \bd \B$ be such that $\x_1 \dashv_{\B}\y_1$ and $\x_2 \dashv_{\B}\y_2$ (and so $\y_1 \dashv_{\B}\x_1$, $\y_2 \dashv_{\B}\x_2$). By choosing $[\x_1 , \x_2 ]_{\B}$ small enough, we can ensure that $[\y_1 , \y_2 ]_{\B}$ is a replacement arc for $\K_{\oo}'$ at $\y$. Let $\mathbf{A}$ be the resulting convex domain as illustrated in Figure \ref{fig:circle2}. Then $\mathbf{A}$ is a smooth strictly convex $A$-domain with circular pieces $[\x_1 , \x_2 ]_{\B}$, $[-\x_1 , -\x_2 ]_{\B}$, $[\y_1 , \y_2 ]_{\B}$ and $[-\y_1 , -\y_2 ]_{\B}$. Clearly, $h(\mathbf{A},\K_{\oo}')\le \epsilon$. Furthermore, we can make the length of $\bd \mathbf{A}\cap \K_{\oo}'$ as close to the length of $\bd\K_{\oo}'$ as we like. Therefore, $h(\mathbf{A}',\K_{\oo})\le \epsilon$  and we can make the length of $\bd \mathbf{A}'\cap \K_{\oo}$ as close to the length of $\bd\K_{\oo}$ as we like.
\end{proof}

In section \ref{sec:hadwiger}, we found that the separable Hadwiger number remains constant over the class of smooth convex domains. Corollary \ref{a-dom} shows that the maximum separable contact number of any packing of $n$ translates of an affine image of a smooth $A$-domain is the same as the corresponding number for the circular disk. Secondly, we observed that affine images of smooth strictly convex $A$-domains form a dense subset of the space of smooth $\oo$-symmetric strictly convex domains under Hausdorff metric. It is, therefore, natural to ask if the maximum separable contact number also remains constant over the set of all smooth strictly convex domains (as we can drop $\oo$-symmetry due to Lemma \ref{minkowski}). We conclude this section with a proof of part (B) of Theorem \ref{mega}, showing that this is indeed the case.

\begin{corollary}\label{strictly}
Let $\K$ be a smooth strictly convex domain in $\E^2$ and $n\geq 2$. Then  $\csep (\K,n,2) = \floor{2n-2\sqrt{n}}$.
\end{corollary}

\begin{proof}
By Lemma \ref{minkowski}, it suffices to prove the result for smooth $\oo$-symmetric strictly convex domains. Let $\K_{\oo}$ be such a domain. By Remark \ref{smooth-contact}, the lower bound $\csep(\K_{\oo})\ge \floor{2n-2\sqrt{n}}$ holds. Let $\cal{P}$ be a maximal contact totally separable packing of $n$ translates of $\K_{\oo}$ and $\cal{H}$ be a finite set of lines in $\E^2$ disjoint from the interiors of the translates in $\cal{P}$ such that any two translates are separated by at least one line in $\cal{H}$. We will construct a smooth strictly convex $A$-domain $\mathbf{A}$ such that $\csep(\K_{\oo} , n, 2) \le \csep(\mathbf{A}', n, 2)$, for $\mathbf{A}'=T(\mathbf{A})$, where $T:\E^2 \to \E^2$ is a properly chosen invertible linear transformation. 

Let $\K_{\oo}+\mathbf{c}\in \cal{P}$. By strict convexity of $\K_{\oo}$, there exist finitely many points $\mathbf{c}_1 \ldots, \mathbf{c}_m \in \bd(\K+\mathbf{c})$ where $\K_{\oo}+\mathbf{c}$ touches other translates in $\cal{P}$ and the lines in $\cal{H}$. Then we call the points $\mathbf{c}_i - \mathbf{c} \in \bd \K_{\oo}$, $i=1,\ldots, m$, the {\it contact positions} on $\K_{\oo}$ corresponding to $\K_{\oo}+\x \in \cal{P}$. Let ${\rm Con}(\cal{P})$ denote the set of all contact positions on $\K_{\oo}$ corresponding to all the translates in $\cal{P}$. 

Let $\x$, $\y$, $\mathbf{P}$ and $\B$ be as in the proof of Theorem \ref{dense}. Using Theorem \ref{dense}, construct a smooth strictly convex $A$-domain $\mathbf{A}$ with circular pieces $c=[\x_1 , \x_2 ]_{\mathbf{A}}= [\x_1 , \x_2 ]_{\B}$, $-c=[-\x_1 , -\x_2 ]_{\mathbf{A}}=[-\x_1 , -\x_2 ]_{\B}$, $c'=[\y_1 , \y_2 ]_{\mathbf{A}}=[\y_1 , \y_2 ]_{\B}$ and $-c'=[-\y_1 , -\y_2 ]_{\mathbf{A}}=[-\y_1 , -\y_2 ]_{\B}$. 

Using Theorem \ref{dense}, we can make $c=[\x_1 , \x_2 ]_{\mathbf{A}}$ sufficiently small so that 
\begin{equation}\label{eq1}
{\rm Con}({\cal{P}})\subseteq (\bd \mathbf{A}' \cap\bd\K_{\oo}) \cup \{\pm T(\x), \pm T(\y)\}.
\end{equation}
(Recall that both $\B$ and $\K_{\oo}' = T^{-1}(\K_{\oo})$ are supported by the sides of $\mathbf{P}$ at the points $\{\pm \x, \pm \y\}$.) Thus from \eqref{eq1}, the arrangement obtained by replacing each translate in $\cal{P}$ by the corresponding translate of $\mathbf{A}'$ is a totally separable packing with at least $\csep(\K_{\oo} , n, 2)$ contacts. 
\end{proof}


\bigskip


\noindent K\'aroly Bezdek \\
\small{Department of Mathematics and Statistics, University of Calgary, Canada}\\
\small{Department of Mathematics, University of Pannonia, Veszpr\'em, Hungary\\
\small{E-mail: \texttt{bezdek@math.ucalgary.ca}}


\bigskip

\noindent Muhammad A. Khan \\
 \small{Department of Mathematics and Statistics, University of Calgary, Canada}\\
 \small{E-mail: \texttt{muhammkh@ucalgary.ca}}


\bigskip

\noindent Michael Oliwa \\
 \small{Department of Mathematics and Statistics, University of Calgary, Canada}\\
 \small{E-mail: \texttt{moliwa@ucalgary.ca}}

\end{document}